\documentclass[11pt]{amsart}

 \usepackage{graphicx,color}
\usepackage{geometry}
\usepackage{color, graphicx}
\usepackage[all]{xy}

\usepackage{calrsfs}
\usepackage[T1]{fontenc} 
\usepackage{textcomp}
\usepackage{times}
 \usepackage[scaled=0.92]{helvet} 

\newtheorem{introtheorem}{Theorem}
\theoremstyle{plain}
\swapnumbers
\newtheorem{thm}{Theorem}[section]
\newtheorem{prop}[thm]{Proposition}
\newtheorem{cor}[thm]{Corollary}
\newtheorem{lem}[thm]{Lemma}
\theoremstyle{definition}
\newtheorem{df}[thm]{Definition}
\newtheorem{notation}[thm]{Notation}

\newtheorem{remark}[thm]{Remark}
\newtheorem{se}[thm]{}
\newtheorem*{remark*}{Remark}

\usepackage{amsfonts}
\usepackage{amssymb}

\renewcommand{\tilde}{\widetilde}
\renewcommand{\bar}{\overline}
\newcommand{\isomto}{\overset{\sim}{\rightarrow}}
\newcommand{\ab}{\mbox{{\tiny \textup{ab}}}}

\DeclareMathOperator{\KMS}{\mathrm{KMS}}
\DeclareMathOperator{\Z}{\mathbb{Z}}
\DeclareMathOperator{\A}{\mathbf{A}}

\DeclareMathOperator{\cH}{\mathcal{H}}

\DeclareMathOperator{\cE}{\mathcal{E}}
\DeclareMathOperator{\cC}{\mathcal{C}}
\DeclareMathOperator{\cO}{\mathcal{O}}
\DeclareMathOperator{\cM}{\mathcal{M}}

\DeclareMathOperator{\Q}{\mathbb{Q}}
\DeclareMathOperator{\R}{\mathbb{R}}
\DeclareMathOperator{\C}{\mathbb{C}}
\DeclareMathOperator{\F}{\mathbb{F}}

\DeclareMathOperator{\Gal}{\mathrm{Gal}}
\DeclareMathOperator{\End}{\mathrm{End}}
\DeclareMathOperator{\tr}{\mathrm{trace}}
\DeclareMathOperator{\p}{\mathfrak{p}}
\DeclareMathOperator{\q}{\mathfrak{q}}

\DeclareMathOperator{\f}{\mathfrak{f}}

\DeclareMathOperator{\GL}{\mathrm{GL}}
\DeclareMathOperator{\Aut}{\mathrm{Aut}}
\DeclareMathOperator{\Hom}{\mathrm{Hom}}
\renewcommand{\Im}{\mathsf{Im}}
\DeclareMathOperator{\bK}{\mathbb{K}}
\DeclareMathOperator{\bL}{\mathbb{L}}
\DeclareMathOperator{\bQ}{\mathbb{Q}}
\DeclareMathOperator{\bM}{\mathbb{M}}
\DeclareMathOperator{\bN}{\mathbb{N}}
\DeclareMathOperator{\fn}{\mathfrak{n}}
\DeclareMathOperator{\fm}{\mathfrak{m}}
\DeclareMathOperator{\fa}{\mathfrak{a}}
\DeclareMathOperator{\fb}{\mathfrak{b}}
\DeclareMathOperator{\fc}{\mathfrak{c}}

\DeclareMathOperator{\fu}{\mathfrak{u}}

\DeclareMathOperator{\fr}{\mathfrak{r}}

\usepackage{bbm}
\DeclareMathOperator{\one}{\mathbbm{1}}

\begin{document}

\date{\today\ (version 5.0)} 
\title[QSM, $L$-series and Anabelian Geometry]{Quantum Statistical Mechanics, \\ $L$-series and Anabelian Geometry}
\author[G.~Cornelissen]{Gunther Cornelissen}
\address{\normalfont Mathematisch Instituut, Universiteit Utrecht, Postbus 80.010, 3508 TA Utrecht, Nederland}
\email{g.cornelissen@uu.nl}
\author[M.~Marcolli]{Matilde Marcolli}
\address{\normalfont Mathematics Department, Mail Code 253-37, Caltech, 1200 E.\ California Blvd.\ Pasadena, CA 91125, USA}
\email{matilde@caltech.edu}


\begin{abstract} \noindent It is known that two number fields with the same Dedekind zeta function are not necessarily isomorphic. The zeta function of a number field can be interpreted as the partition function of an associated quantum statistical mechanical system, which is a $C^*$-algebra with a one parameter group of automorphisms, built from Artin reciprocity. In the first part of this paper, we prove that isomorphism of number fields is the same as isomorphism of these associated systems. Considering the systems as noncommutative analogues of topological spaces, this result can be seen as another version of Grothendieck's ``anabelian'' program, much like the Neukirch-Uchida theorem characterizes isomorphism of number fields by topological isomorphism of their associated absolute Galois groups. 

In the second part of the paper, we use these systems to prove the following. If there is a group isomorphism $  \psi \, : \,  \widehat G_{\bK}^{\ab} \isomto \widehat G_{\bL}^{\ab}$ between the character groups (viz., Pontrjagin duals) of the \emph{abelianized} Galois groups of the two number fields  that induces an equality of \emph{all} corresponding $L$-series
$ L_{\bK}(\chi,s) = L_{\bL}(\psi(\chi),s) $ (not just the zeta function), then the number fields are isomorphic. 

{This is also equivalent to the purely algebraic statement that there exists an isomorphism $\psi$ as a above and a norm-preserving group isomorphism between the ideals of $\bK$ and $\bL$ that is compatible with the Artin maps via $\psi$.}  \end{abstract}

\thanks{We thank Bart de Smit, Eugene Ha, Fumiharu Kato, Bram Mesland, Sergey Neshveyev, Jorge Plazas, Florian Pop, Simen Rustad and Bora Yalkinoglu for their  comments, and especially Hendrik Lenstra for his meticulous criticism of some of the number theory in a previous version.} 
\maketitle

\tableofcontents 

\section*{Introduction}

Can one describe isomorphism of two number fields $\bK$ and $\bL$ from associated analytic or topological objects? Here are some attempts (``no''-answers indexed by \textbf{N}; ``yes''-answers by \textbf{Y}):
\begin{enumerate}
\item[\textbf{(N1)}] An \textbf{equality of their Dedekind zeta functions} (so-called \emph{arithmetic equivalence}) does not imply that $\bK$ and $\bL$ are isomorphic, as was shown by Ga{\ss}mann (\cite{G}, cf.\ also Perlis \cite{Perlis1}, or \cite{Klingen}). An example is provided by $$\bK=\Q(\sqrt[8]{3}) \mbox{ and } \bL=\Q(\sqrt[8]{3 \cdot 2^4})$$ (\cite{Perlis1}, \cite{Komatsu}). However, the implication is true  \emph{if} $\bK$ and $\bL$ are Galois over $\bQ$ (Theorem of Bauer \cite{Bauer1} \cite{Bauer2}, nowadays a corollary of Chebotarev's density theorem, see, e.g., Neukirch \cite{Neukirch} 13.9).
\item[\textbf{(N2)}] An \textbf{isomorphism of their adele rings} $\A_{\bK}$ and $\A_{\bL}$ as topological rings does not imply that $\bK$ and $\bL$ are isomorphic, cf.\ Komatsu (\cite{Komatsu2}). An example is $$\bK=\Q(\sqrt[8]{2 \cdot 9}) \mbox{ and } \bL=\Q(\sqrt[8]{2^5 \cdot 9}).$$  An adelic isomorphism does imply in particular an equality of the zeta functions of $\bK$ and $\bL$, but is not equivalent to it --- the example in (\textbf{N1}) has non-isomorphic adele rings, cf.\ \cite{Komatsu}. However, for a global function field adelic isomorphism and arithmetic equivalence is the same, cf.\ Turner \cite{Turner}.
\item[\textbf{(N3)}] \label{abab} An \textbf{isomorphism of the Galois groups of the maximal abelian extensions} $G^{\ab}_{\bK}$ and $G^{\ab}_{\bL}$ as topological groups does not imply an isomorphism of the fields $\bK$ and $\bL$. For example, $$\bK=\Q(\sqrt{-2}) \mbox{ and } \bL=\Q(\sqrt{-3})$$ have isomorphic abelianized absolute Galois groups (see Onabe \cite{Onabe}).
\end{enumerate}
However \dots
\begin{enumerate} 
\item[\textbf{(Y1)}] An \textbf{isomorphism of their absolute Galois groups} $G_{\bK}$ and $G_{\bL}$ as topological groups implies isomorphism of the fields $\bK$ and $\bL$: this is the celebrated theorem of Neukirch and Uchida (In \cite{NeukirchInv}, Neukirch proved this for fields that are Galois over $\Q$; in \cite{U3}, Uchida proved the general case,  cf.\ also \cite{NBook} 12.2, Ikeda \cite{Ikeda} and unpublished work of Iwasawa). It can be considered the first manifestation (zero-dimensional case) of the so-called ``anabelian'' philosophy of Grothendieck (\cite{Gro}, esp.\ footnote (3)): the neologism ``anabelian'' seems to have been  coined by Grothendieck by contrast with statement \textbf{(N3)} above.
\item[\textbf{(Y2)}] In an unpublished work, Richard Groenewegen \cite{Groen} proved a \textbf{Torelli theorem} for number fields: if two number fields have ``strongly monomially equivalent'' $h^0$-function in Arakelov theory (in the sense of van der Geer and Schoof, cf.\ \cite{GS}), then they are isomorphic.
\end{enumerate}

The starting point for this study is the observation that the zeta function of a number field $\bK$ can be realized as the partition function of a quantum statistical mechanical (QSM) system in the style of Bost and Connes (cf.\ \cite{BC} for $\bK=\bQ$). The QSM-systems for general number fields that we consider are those that were constructed by Ha and Paugam (see section 8 of \cite{HP}, which is a specialization of their more general class of QSM-systems associated to Shimura varieties), and further studied by Laca, Larsen and Neshveyev in \cite{LLN}. 
This quantum statistical mechanical system consists of a $C^*$-algebra $A_{\bK}$ (the noncommutative analogue of a topological space) with a time evolution $\sigma_{\bK}$ (i.e., a continuous group homomorphism $\R \rightarrow \Aut{A_{\bK}}$) --- for the exact definition, see Section \ref{s2} below, but the structure of the algebra is 
$$
A_{\bK}:=C(X_{\bK})\rtimes J^+_{\bK}, \mbox{ with } X_{\bK}:=G^{\ab}_{\bK}\times_{\hat\cO_{\bK}^*} \hat\cO_{\bK},
$$
where $\hat\cO_{\bK}$ is the ring of finite integral adeles and $J^+_{\bK}$ is the semigroup of ideals, which acts on the space $X_{\bK}$ by Artin reciprocity. The time evolution is only non-trivial on elements $\mu_{\fn} \in A_{\bK}$ corresponding to ideals $\fn \in J^+_{\bK}$, where it acts by multiplication with the norm $N(\fn)^{it}$. We also need the (non-involutive) dagger-subalgebra $A_{\bK}^\dagger$ generated algebraically by functions in $C(X_{\bK})$ and the partial isometries $\mu_{\fn}$ for $\fn \in J_{\bK}^+$ (but \emph{not} $\mu_{\fn}^*$; such non-self adjoint algebras and their closures have been considered before in connection with the reconstruction of dynamical systems up to (piecewise) conjugacy, see e.g.\ \cite{Davidson}). 

For now, it is important to notice that the structure involves the abelianized Galois group and the adeles, but not the absolute Galois group. In this sense, it is ``not anabelian''; but of course, it is ``noncommutative'' (in noncommutative topology, the crossed product construction is an analog of taking quotients). In light of the previous discussion, it is now natural to ask whether the QSM-system (which contains simultaneously the zeta function from \textbf{(N1)}, a topological space built out of the adeles from \textbf{(N2)} and the abelianized Galois group from \textbf{(N3)}) does characterize the number field. 

We call two general QSM-systems  \emph{isomorphic} if there is a $C^*$-algebra isomorphism between the algebras that intertwines the time evolutions. Our main result is that the QSM-system cancels out the defects of \textbf{(N1)---(N3)} in exactly the right way:

\begin{introtheorem} \label{main}
Let $\bK$ and $\bL$ denote arbitrary number fields. Then the following conditions are equivalent:
 \begin{enumerate}
\item[\textup{(}i{)}] \textup{[Field isomorphism]} $\bK$ and $\bL$ are isomorphic as fields;
\item[\textup{(}ii{)}]  \textup{[QSM isomorphism]} there is an isomorphism $\varphi$ of QSM systems $(A_{\bK},\sigma_{\bK})$ and $(A_{\bL},\sigma_{\bL})$ that respects the dagger subalgebras: $\varphi(A_{\bK}^\dagger)=A_{\bL}^\dagger$. 
\end{enumerate}
\end{introtheorem}

One may now ask whether the ``topological'' isomorphism from (ii) can somehow be captured by an analytic invariant, such as the Dedekind zeta function, which in itself doesn't suffice. Our second main theorem says that this is indeed the case:

\begin{introtheorem} \label{main2}
Let $\bK$ and $\bL$ denote arbitrary number fields. Then the following conditions are equivalent: 
 \begin{enumerate}
 \item[\textup{(}i{)}] \textup{[Field isomorphism]} $\bK$ and $\bL$ are isomorphic as fields;
\item[\textup{(}iii{)}] \textup{[L-isomorphism]} there is group isomorphism between (the Pontrjagin duals of) the abelianized Galois groups $$\psi \, : \, \widehat{G}_{\bK}^{\ab} \isomto \widehat{G}_{\bL}^{\ab}$$ such that for every character $\chi \in \widehat{G}_{\bK}^{\ab}$, we have an identification of $L$-series {for these generalized Dirichlet characters}
$$ L_{\bK}(\chi,s) = L_{\bL}(\psi(\chi),s). $$ 
\end{enumerate}
\end{introtheorem}

Condition (iii) can be considered as the correct generalization of arithmetic equivalence (which is (iii) for the trivial character only)  to an analytic equivalence that \emph{does} capture isomorphism. It should also be observed at this point that (Hecke) $L$-series occur naturally in the description of generalized equilibrium states (KMS-states) of the QSM-system, and this is how we originally discovered the statement of the theorem. 

{Finally, there is the following purely algebraic reformulation, which upgrades \textbf{(N3)} by adding a certain compatibility of the isomorphism of abelianized Galois groups with ramification: 

\begin{introtheorem} \label{main3}
Let $\bK$ and $\bL$ denote arbitrary number fields. Then the following conditions are equivalent:
 \begin{enumerate}
\item[\textup{(}i{)}] \textup{[Field isomorphism]} $\bK$ and $\bL$ are isomorphic as fields;
\item[\textup{(}iv{)}] \textup{[Reciprocity isomorphism]} there is a topological group isomorphism $$\hat{\psi} \, : \, {G}_{\bK}^{\ab} \isomto {G}_{\bL}^{\ab}$$ and an isomorphism $$\Psi \, : \, J_{\bK}^+ \isomto J_{\bL}^+$$ of semigroups of ideals such that the following two compatibility conditions are satisfied: 
\begin{enumerate}
\item[\textup{(a)}] compatibility of $\Psi$ with norms: $N_{\bL}(\Psi(\fn))=N_{\bK}(\fn)$ for all ideals $\fn \in J_{\bK}^+$; and
\item[\textup{(b)}] compatibility with the Artin map: for every finite abelian extension $$\bK'=\left(\bK^{\ab}\right)^N/\bK$$ (with $N$ a subgroup in $G_{\bK}^{\ab}$) and every prime $\p$ of $\bK$ unramified in $\bK'$, the prime $\Psi(\p)$ is unramified in the corresponding field extension $$\bL':=\left(\bL^{\ab}\right)^{\hat\psi(N)}/\bL,$$ and we have
$$ \hat\psi \left( \mathrm{Frob}_{\p} \right) = \mathrm{Frob}_{\Psi(\p)}. $$  
\end{enumerate}
\end{enumerate}
\end{introtheorem}
}

We first say a few words about the proofs. We start by proving that QSM-isomorphism (ii) implies field isomorphism (i). For this, we first prove that the fields are arithmetically equivalent (by interpreting the zeta functions as partition functions and studying the relation between the Hamiltonians for the two systems), and then we use some results from the reconstruction of dynamical systems from non-involutive algebras to deduce an identification of the semigroups of integral ideals of $\bK$ and $\bL$ and a compatible homeomorphism of $X_{\bK}$ with $X_{\bL}$. We use this to prove that $\varphi$ preserves a layered structure in the algebra corresponding to ramification in the field, and this allows us to prove that there is a homomorphism of $G_{\bK}^{\ab}$ with $G_{\bL}^{\ab}$ ``compatible with the Artin map'', and an isomorphism of unit ideles (built up locally from matching of inertia groups), and finally, multiplicative semigroups of the totally positive elements (viz., positive in every real embedding of the number field) of the rings of integers, which occur as inner endomorphisms of the dagger-subalgebra. We then prove that the map is additive modulo large enough inert primes, using the Teichm\"uller lift.  Finally, it is easy to pass from an isomorphism of semirings of totally positive elements to an isomorphism of the fields. 

Then we prove that L-isomorphism (iii) implies QSM-isomorphism (ii): from the matching of $L$-series, we get a matching of semigroups of ideals, compatible with the Artin map, by doing some character theory with the $L$-series of the number fields as counting functions of ideals that have a given norm and a given image under the Artin map in the maximal abelian extension where they remain unramified. We then extend these maps to the whole algebra by a topological argument. 

In this context, one may try to rewrite the main theorems in a functorial way, as a bijection of certain Hom-sets. It would be interesting to understand the relation to the functor from number fields to QSM-systems in \cite{LNT}. 

It is easy to see that reciprocity isomorphism (iv) implies L-isomorphism (iii), and of course, field isomorphism (i) implies the rest. 

The proof seems to indicate that a mere isomorphism of the $C^*$-algebras $A_{\bK}$ and $A_{\bL}$ does not suffice to conclude that $\bK$ and $\bL$ are isomorphic; we make heavy use of the compatibility with time evolution given by the norms. It would be interesting to know whether one can leave out from QSM-isomorphism the condition of preserving the dagger subalgebra. Neshveyev has shown us an example of a (non-dagger) inner endomorphism of $(A_{\bK},\sigma_{\bK})$ that doesn't respect $C(X_{\bK})$. On the other hand, QSM-isomorphism does imply arithmetic equivalence, so by Ga{\ss}mann's results, QSM-isomorphism (without requiring dagger isomorphism) for Galois extensions of $\Q$ already implies field isomorphism. 

Finally, we remark that our proof is constructive: we exhibit, from the various other isomorphisms, an explicit field isomorphism. 

\begin{remark*} We make a few remarks about the condition of L-isomorphism in the theorem. First of all, the equivalence of field isomorphism and L-isomorphism/reciprocity isomorphism is a purely number theoretical statement, without reference to QSM-systems. It is a number theoretical challenge to provide a direct proof of this equivalence (of course, one can clear the current proof of QSM-lingo).

Secondly, one may wonder whether the L-isomorphism condition (iii) can be replaced by something weaker. As we already observed, requiring (iii) for the trivial character only is not enough, but what about,  for example, this condition:
\begin{quote}
\textup{(}iii{)}$_2$ \emph{All rational quadratic $L$-series of $\bK$ and $\bL$ are equal, i.e.\ for all integers $d$ that are not squares in $\bK$ and $\bL$, we have $L_{\bK}(\chi_d,s)=L_{\bL}(\chi_d,s)$.}
\end{quote}
By considering only rational characters, one does not need to introduce a bijection of abelianized Galois groups, since there is an automatic matching of conductors. One can also consider a similar statement (iii)$_n$ for all $n$-th order rational $L$-series. 

It turns out that (iii)$_2$ is \emph{not} equivalent to (ii). We prove that as soon as $\bK$ and $\bL$ have the same zeta functions, condition (iii)$_2$ holds (the proof uses \emph{Ga{\ss}mann-equivalence}, and was discovered independently by Lotte van der Zalm in her undergraduate thesis \cite{Lotte}).  Another number theoretical challenge is to give a purely analytical proof of this statement (i.e., not using group theory).

Finally, we note that the condition of L-isomorphism is motivic: it gives an identification of  $L$-series of rank one motives over both number fields (in the sense of \cite{Deligne}, \S 8).
\end{remark*}

\begin{remark*} After announcing our result at the GTEM conference in Barcelona (september 2010), Bart de Smit rose to the first number theoretical challenge (to prove the equivalence of field isomorphism and L-isomorphism), by using Galois theory, cf.\ \cite{BDS}. The method of de Smit allowed him to prove that if for two number fields $\bK$ and $\bL$, the sets of zeta functions of all their abelian extension are equal, then the fields are isomorphic. He can also prove that it suffices for this conclusion to hold that there is a bijection between the $2$-torsion subgroups of $\widehat G_{\bK}^{\ab}$ and $\widehat G_{\bL}^{\ab}$ (so the sets of all quadratic or trivial characters) such that the corresponding $L$-series are equal, and for given fields, one can construct a finite list of quadratic characters which it suffices to check. Also, with Hendrik Lenstra, he has proven that every number field has an abelian $L$-series that does not occur for any other number field. 
\end{remark*}

\begin{remark*}[Anabelian vs.\ noncommutative]
The anabelian philosophy is, in the words of Gro\-then\-dieck (\emph{Esquisse d'un programme}, \cite{Gro}, footnote (3)) ``a construction which pretends to ignore [\dots] the algebraic equations which traditionally serve to describe schemes, [\dots] to be able to hope to reconstitute a scheme [\dots] from [\dots] a purely topological invariant [\dots]''.  
In the zero-dimensional case, the fundamental group plays no r\^ole, only the absolute Galois group, and we arrive at the theorem of Neukirch and Uchida (greatly generalized in recent years, notably by Bogomolov-Tschinkel \cite{BoTs}, Mochizuki \cite{Moch} and Pop \cite{Pop}, compare \cite{Sz}). 

Our main result indicates that QSM-systems for number fields can be considered as some kind of substitute for the absolute Galois group. The link to Grothendieck's proposal arises via a philosophy from noncommutative geometry that ``topology = $C^*$-algebra'' and ``time evolution = Frobenius''. This would become a genuine analogy if one could unearthen a ``Galois theory'' that describes a categorical equivalence between number fields on the one hand, and their QSM-systems on the other hand. Anyhow, it seems Theorem \ref{main} indicates that one may, in some sense, substitute ``noncommutative'' for ``anabelian''.\footnote{Interestingly, the Wikipedia entry for ``Anabelian geometry'' starts with ``Not to be confused with Noncommutative Geometry'' (retrieved 16 Aug 2010).} This substitution has an interesting side effect: in the spirit of Kronecker's programme, one wants to characterize a number field by structure that is ``internal'' to it (i.e., not using extensions): this is the case for the QSM-system, since class field theory realizes Kronecker's programme for abelian extensions. On the other hand, anabelian geometry characterizes a number field by its absolute Galois group, an object whose ``internal'' understanding remains largely elusive and belongs to the Langlands programme. 

In the style of Mochizuki's \emph{absolute} version of anabelian geometry (cf.\ \cite{Mo}), one may ask how to reconstruct a number field from its associated QSM-system (or $L$-series), rather than to reconstruct an isomorphism of fields from an isomorphism of QSM-systems (or an L-isomorphism). 

It would be interesting to study the analogue of our results for the case of function fields, and higher dimensional schemes. Jacob \cite{Jacob} and Consani-Marcolli \cite{ConsM} have constructed function field analogues of QSM systems that respectively have the Weil and the Goss zeta function as partition function. The paper \cite{CKZ} studies arithmetic equivalence of function fields using the Goss zeta function. 
\end{remark*}

\begin{remark*}[Link with hyperring theory]
Connes and Consani have studied the adele class space as a hyperring in the sense of Krasner (\cite{Krasner}). They prove in \cite{CC} (Theorem 3.13) that
\begin{quote} (v) [Hyperring isomorphism] \ \emph{the two adele class spaces  $\A_{\bK}\! /\! \bK^* \cong \A_{\bL}\! /\! \bL^*$ are isomorphic as hyperrings over the Krasner hyperfield;}
\end{quote} is equivalent to field isomorphism. The proof is very interesting: it uses classification results from incidence geometry.  One may try to prove that QSM-isomorphism implies hyperring isomorphism directly (thus providing a new proof of the equivalence of field isomorphism with QSM-isomorphism; this is especially tempting, since  Krasner developed his theory of hyperrings for applications to class field theory).   

Observe that the equivalence of hyperring isomorphism with field isomorphism is rather far from the anabelian philosophy (which would be to describe algebra by topology), since it uses (algebraic) isomorphism of hyperrings to deduce isomorphism of fields. But it might be true that the \emph{topology/geometry} of the hyperring can be used instead. As a hint, we refer to Theorem 7.12 in  \cite{CC}: over a global function field, the groupoid of prime elements of the hyperring of adele classes \emph{is} the abelianized loop groupoid of the curve, cf.\ also \cite{CC2}, Section 9.  
\end{remark*}

\begin{remark*}[Analogues in Riemannian geometry]
There is a well-known (limited) analogy between the theory of $L$-series in number theory and the theory of spectral zeta functions in Riemannian geometry. For example, the ideas of Ga{\ss}mann were used by Sunada to construct isospectral, non-isometric manifolds (cf.\ \cite{Sunada}): the spectral zeta function does not determine a Riemannian manifold up to isometry (actually, not even up to homeomorphism).  

In \cite{Riem}, it was proven that the isometry type of a closed Riemannian manifold is determined by a \emph{family} of Dirichlet series associated to the Laplace-Beltrami operator on the manifold. In \cite{CMRiem}, it was proven that one can reconstruct a compact hyperbolic Riemann surface from a suitable \emph{family} of Dirichlet series associated to a spectral triple. These can be considered as analogues in manifold theory of the equivalence of (i) and (iii).

One might consider as another analogy of (iii) the matching of all $L$-series of Riemannian coverings of two Riemannian manifolds, but this appears not to be entirely satisfactory; for example, there exist simply connected isospectral, non-isometric Riemannian manifolds (cf.\ Sch\"uth \cite{Schueth}).  

One may consider Mostow rigidity (a hyperbolic manifold of dimension at least three is determined by its fundamental group) as an analogue of the anabelian theorem. Again, this is very \emph{an}abelian, since the homology rarely determines a manifold. 

There is a further occurence of $L$-series in geometry (as was remarked to us by Atiyah): the Riemann zeta function is the only Dedekind zeta function that occurs as spectral zeta function of a manifold (namely, the circle); but more general $L$-series can be found in the geometry of the resolution of the cusps of a Hilbert modular variety (\cite{Atiyah}, compare \cite{Solv}), a kind of ``virtual manifold'' that also has a ``quotient structure'', just like the QSM-system algebra is a noncommutative  quotient space. 
\end{remark*}

\section*{Disambiguation of notations}

There will be one notational sloppiness throughout: we will denote maps that are induced by a given isomorphism $\varphi$ by the same letter $\varphi$. 

Since the number theory and QSM literature have conflicting standard notations, we include a table of notations for the convenience of the reader:

\medskip

\begin{footnotesize}

\noindent $R^*$ \dotfill invertible elements of a ring $R$\\
\noindent $R^\times$ \dotfill non-zero elements of a ring $R$\\

\noindent $\widehat{G}$ \dotfill Pontrjagin dual: continuous $\Hom(G,S^1)$ of a topological abelian group $G$ \\
\noindent $G^0$ \dotfill connected component of identity \\

\noindent $\bK, \bL, \bM, \bN$ (blackboard bold capitals) \dotfill number fields\\
\noindent $L_{\bK}(\chi,-)$ \dotfill $L$-series of field $\bK$ for generalized Dirichlet character $\chi\in \widehat G_{\bK}^{\ab}$\\

\noindent $J_{\bK}^+$ \dotfill semigroup of integral ideals of a number field $\bK$ \\
\noindent $N=N_{\bK}=N_{\Q}^{\bK}$ \dotfill the norm map  on ideals of the number field $\bK$ \\
\noindent $\fn,\p,\q$ (fraktur letters) \dotfill integral ideals of a number field \\

\noindent $\cO_{\bK}$ \dotfill ring of integers of a number field $\bK$\\
\noindent $\cO_{\bK,+}$ \dotfill semiring of totally positive integers of a number field $\bK$\\

\noindent $\hat\cO_{\bK,\p}$ \dotfill completed local ring of $\p$-adic integers in $\bK$\\
\noindent $\hat\cO_{\bK}$ \dotfill ring of finite integral adeles of a number field $\bK$\\

\noindent $\bar{\bK}_{\p}$ \dotfill residue field of $\bK$ at $\p$\\
\noindent ${\bK}_{\fn}$ \dotfill maximal abelian extension of $\bK$ unramified outside prime divisors of $\fn$\\

\noindent $f(\p|p)=f(\p|{\bK})$ \dotfill inertia degree of $\p$ over $p$, in $\bK$ \\ 
\noindent $\f_\chi$ \dotfill conductor of $\chi$ \\
\noindent $f_\chi$ \dotfill element of $A_{\bK}$ that implements the character $\chi$ \\
\noindent $f_{\chi,\fm}$ \dotfill generator of $C(X_{\bK})$ as in Lemma \ref{gengen} \\

\noindent $G_{\bK}$ \dotfill absolute Galois group of $\bK$ \\
\noindent $G_{\bK}^{\ab}$ \dotfill Galois group of maximal abelian extension of $\bK$ \\
\noindent $G_{\bK,\fn}^{\ab}$ \dotfill Galois group of maximal abelian extension of $\bK$ unramified at divisors of $\fn$ \\
\noindent $\mathring{G}_{\bK,\fn}^{\ab}$ \dotfill Galois group of maximal abelian extension of $\bK$ unramified outside divisors of $\fn$ \\

\noindent $\A_{\bK}$ \dotfill adele ring of a number field $\bK$\\
\noindent $\A_{\bK,f}$ \dotfill finite (non-archimedean) part of the adele ring of a number field $\bK$\\
\noindent $A_{\bK}$ \dotfill the $C^*$ algebra of the QSM-system of the number field $\bK$\\
\noindent $\vartheta_{\bK}$ \dotfill Artin reciprocity map $\A_{\bK}^* \rightarrow G_{\bK}^{\ab}$ \\

\noindent $\beta$ \dotfill positive real number representing ``inverse temperature''\\
\noindent $X_{\bK}$ \dotfill topological space of $[(\gamma,\rho)] \in G_{\bK}^{\ab} \times_{\hat\cO_{\bK}^*} \hat\cO_{\bK}$ underlying part of the algebra $A_{\bK}$ \\ 
\noindent $X^1_{\bK}$ \dotfill dense subspace of $[(\gamma,\rho)] \in X_{\bK}$ on which none of components of $\rho$ is zero \\ 
\noindent $\mu_{\fn}$ \dotfill element of the $C^*$-algebra $A_{\bK}$ corresponding to the ideal $\fn \in J_{\bK}^+$\\

\noindent $e_{\fn}$ \dotfill $=\mu_{\fn} \mu_{\fn}^*$, projector\\
\noindent $\epsilon_\gamma$ \dotfill symmetry of $A_{\bK}$ induced by multiplication, for $\gamma  \in G_{\bK}^{\ab}$, with $[(\gamma,1)]$ on $X_{\bK}$\\
\noindent $\varepsilon_s$ \dotfill endomorphism of $A_{\bK}$ given by $\varepsilon_s(f)(\gamma,\rho)=f(\gamma,s^{-1}\rho)e_{s\hat\cO_{\bK} \cap \bK}$ \\

\noindent $\rho$ \dotfill finite integral adele $\in \hat\cO_{\bK}$\\
\noindent $\rho_{\p}$ \dotfill $\p$-component of an adele $\rho$\\
\noindent $\rho_{\fn}(f)$ \dotfill action of ideal $\fn$ on $f \in C(X_{\bK})$ : $\rho_{\fn}(f)=f(\vartheta_{\bK}(\fn) \gamma, \fn^{-1} \rho)e_{\fn}$\\
\noindent $\fn \ast x$ \dotfill action of ideal $\fn$ on $x \in X_{\bK}$ : $\fn \ast [(\gamma,\rho)]=[(\vartheta_{\bK}(\fn)^{-1} \gamma, \fn \rho)]$\\
\noindent $\sigma_{\fn}(f)$ \dotfill partial inverse to $\rho_{\fn}$ : $\sigma_{\fn}(f)=f(\fn \ast \rho)$ \\ 
\noindent $\sigma_{\bK}=\sigma_t=\sigma_{\bK,t}$ \dotfill the time evolution (in time $t$) of the QSM-system of the number field $\bK$\\
\noindent $\rtimes$ \dotfill crossed product construction of $C^*$-algebras (not semidirect product of groups) \\

\noindent $\omega$ \dotfill a state of a $C^*$-algebra \\
\noindent $\omega_\beta$ \dotfill a $\KMS_\beta$ state of a $C^*$-algebra \\
\noindent $\pi_\omega$ \dotfill GNS-representation corresponding to $\omega$ \\
\noindent $\cM_\omega$ \dotfill weak closure of algebra in GNS-representation \\

\noindent $H$ \dotfill Hamiltonian\\
\noindent $\cH$ \dotfill Hilbert space \\

\noindent $\KMS_\beta(A,\sigma)$ \dotfill the set of $\KMS_\beta$-states of the QSM-system $(A,\sigma)$ \\
\noindent $\KMS_\beta(\bK)$ \dotfill $\KMS_\beta(A_{\bK},\sigma_{\bK})$ \\

\end{footnotesize}

\part{QSM-ISOMORPHISM OF NUMBER FIELDS}

\section{Isomorphism of QSM systems} \label{isom}

We recall some definitions and refer to \cite{BR}, \cite{CM2}, and Chapter 3 of \cite{CM}  for more information and for some physics background. After that, we introduce isomorphism of QSM-systems, and prove it preserves $\KMS$-states (cf.\ infra). 

\begin{df}
A \emph{quantum statistical mechanical system} (QSM-system) $(A,\sigma)$ is a (unital) $C^*$-algebra $A$ together with a so-called \emph{time evolution} $\sigma$, which is a continuous group homomorphism $$\sigma \, : \, \R \rightarrow \Aut A \, : \, t \mapsto \sigma_t.$$ 
A \emph{state} on $A$ is a continuous positive unital linear functional $ \omega \, : \, A \rightarrow \C$. We say $\omega$ is a \emph{KMS$_\beta$ state} for some $\beta \in \R_{>0}$ if for all $a,b \in A$, there exists a function $F_{a,b}$, holomorphic in the strip $0<\Im\, z<\beta$ and bounded continuous on its boundary, such that 
$$ F_{a,b}(t)=\omega(a\sigma_t(b)) \mbox{ and } F_{a,b}(t+i\beta)=\omega(\sigma_t(b)a) \ \ \ \ (\forall t \in \R). $$
Equivalently, $\omega$ is a $\sigma$-invariant state with $\omega(ab)=\omega(b\sigma_{i\beta}(a))$ for $a,b$ in a dense set of $\sigma$-analytic elements. The set $\KMS_\beta(A,\sigma)$ of $\KMS_\beta$ states is topologized as a subspace of the convex set of states, a weak* closed subset of the unit ball in the operator norm of bounded linear functionals on the algebra. 
A  $\KMS_\beta$ state is called \emph{extremal} if it is an extremal point in the (compact convex) set of $\KMS_\beta$ states for the weak (i.e., pointwise convergence) topology. 
\end{df}

\begin{remark}[Physical origins]
This notion of QSM-system is one of the possible physical theories of quantum statistical mechanics; one should think of  $A$ as the algebra of observables, represented on some Hilbert space $\cH$ with orthonormal basis $\{\Psi_i\}$; the time evolution, in
the given representation, is generated by a Hamiltonian $H$ by \begin{equation} \label{imp} \sigma_t(a) = e^{itH}ae^{-itH},\end{equation} and (mixed) states of the system are combinations $$a \mapsto \sum \lambda_i \langle \Psi_i | a \Psi_i \rangle$$ which will mostly be of the form $$a \mapsto \tr(\rho a)$$ for some density matrix $\rho$. A typical equilibrium  state (here, this means stable by time evolution) is a Gibbs state $$a \mapsto \tr(a e^{-\beta H})/\tr(e^{-\beta H})$$ at temperature $1/\beta$, where we have normalized by the partition function $$\tr(e^{-\beta H}).$$ The KMS-condition (named after Kubo, Martin and Schwinger) is a correct generalization of the notion of equilibrium state to more general situations, for example when the trace class condition $$\tr(e^{-\beta H})< \infty,$$
needed to
define Gibbs states, no longer holds (cf.\ Haag, Hugenholtz and Winnink \cite{HHW}).
\end{remark}

\begin{remark}[Semigroup crossed product]
We recall the construction of a \emph{semigroup crossed product algebra}.  A semigroup $C^*$-dynamical system is a triple
$(A,S,\rho)$ of a $C^*$-algebra $A$, a semigroup $S$ and
an action $\rho$ of $S$ by endomorphisms of $A$. A
covariant representation $(\pi,\mu)$ is a pair of a
representation $\pi$ of the $C^*$-algebra $A$ as
bounded operators on a Hilbert space $\cH$ and a
representation $\mu$ of the semigroup $S$ on $\cH$
by isometries, with the property that
$$ \pi(\rho_s(a)) = \mu_s \pi(a) \mu_s^* $$
for all $a\in A$ and $s\in S$. Then the crossed product
$C^*$-algebra $A\rtimes_\rho S$ is the universal
$C^*$-algebra such that each covariant representation
$(\pi,\mu)$ factors through a representation of  $A\rtimes_\rho S$.

The existence of $A\rtimes_\rho S$, with an embedding of $A$
in $A\rtimes_\rho S$, is guaranteed when
the semigroup $S$ is an Ore semigroup, namely it is
cancellative ($as=bs$ or $sa=sb$ implies $a=b$ in $S$)
and right-reversible ($Ss\cap St\neq \emptyset$ for all $s,t\in S$),
the action $\rho$ is by injective endomorphisms, which extend
continuously to the multiplier algebra $M(A)$ mapping the identity
to a projection.

Under these same hypotheses on the semigroup $S$ and the action
$\rho$, the algebra $A\rtimes_\rho S$ is the closure of the linear span
of all monomials of the form $\mu_s^* a \mu_t$, with $s,t \in S$ and
$a\in A$, where the $\mu_s$ here denote the isometries in
$A\rtimes_\rho S$ associated to elements $s\in S$. In particular,
the isometries $\mu_s$ satisfy $\mu_s \mu_t = \mu_{st}$ and
$\mu_s^* \mu_s=1$, while $\mu_s \mu_s^*$ is a projector.
One also has the relations $a \mu_s^* = \mu_s^* \rho_s(a)$ and
$\mu_s a = \rho_s(a) \mu_s$.

See \cite{Larsen} for a more detailed discussion of semigroup
crossed product algebras and their relation to partially defined
actions of the associated enveloping group $G =S^{-1} S$
(which exists and is unique up to canonical isomorphism in
the Ore case).
\end{remark}

\begin{df}
The \emph{dagger subalgebra} $B^\dagger$ of the semigroup crossed product $ B=A\rtimes_\rho S$ is the (non-involutive) subalgebra generated  algebraically by $A$ and and $\mu_t$ for $t \in S$ (but not including the $\mu_t^*$). 
\end{df}

What we call ``dagger subalgebra'' (and its closure) can be seen as a noncommutative analogue of the disc algebra; its study was initiated by Arveson and Josephson, for references see, e.g., \cite{Davidson}, \cite{Power}. 

\medskip 

We now introduce the following equivalence relation for QSM-systems: 

\begin{df} \label{dfQSMiso}
An \emph{isomorphism} of two QSM-systems $(A,\sigma)$ and $(B,\tau)$ is a $C^*$-algebra isomorphism $\varphi: A \isomto B$ that intertwines time evolutions, i.e., such that the following diagram commutes:
$$\xymatrix{ A \ar@{->}[r]^{\varphi}_{\sim} \ar@{->}[d]_{\sigma} & B \ar@{->}[d]^{\tau}  \\ 
A \ar@{->}[r]^{\varphi}_{\sim} & B  
     }$$
\end{df}

\begin{df} \label{dfdaggeriso}
If $(A,\sigma)$ and $(B,\tau)$ are two QSM-systems with given dagger-subalgebras $A^\dagger \subseteq A$ and $B^\dagger \subseteq B$ that are preserved by the respective time evolutions (i.e., $\sigma(A^\dagger) \subseteq A^\dagger$ and $\tau(B^\dagger) \subseteq B^\dagger$), then we call an isomorphism $\varphi$ of the two systems a \emph{dagger-isomorphism} if $\varphi(A^\dagger)=B^\dagger$. \end{df}

\begin{lem} \label{basic} Let $\varphi :  (A,\sigma) \isomto (B,\tau)$ denote an isomorphism of QSM systems. Then for any $\beta>0$, \begin{enumerate}
\item[\textup{(}i{)}] pullback $$\varphi^\ast \, : \, \KMS_\beta(B,\tau) \isomto \KMS_\beta(A,\sigma) \, : \,  \omega \mapsto \omega \circ \varphi $$
 is a homeomorphism between the spaces of $\KMS_\beta$ 
states on $B$ and $A$;
 \item[\textup{(}ii{)}]$\varphi^\ast$ induces a homeomorphism between extremal $\KMS_\beta$ states on $B$ and $A$. 
\end{enumerate}
\end{lem}

\begin{proof}
The map $\varphi$ obviously induces a bijection between states on $B$ and states on $A$. 

For (i), let $F_{a,b}$ be the holomorphic function that implements the $\KMS_\beta$-condition for the state $\omega$ on $(B,\tau)$ at  $a,b \in B$, so $$F_{a,b}(t)=\omega(a \tau_t(b)) \mbox{ and } F_{a,b}(t+i\beta)=\omega(\tau_t(b)a).$$ The following direct computation then shows that the function $F_{\varphi(c),\varphi(d)}$ implements the $\KMS_\beta$-condition for the state $\varphi^\ast \omega$ on $(A,\sigma)$ at  $c,d \in A$:
$$(\omega \circ \varphi)(c \sigma_t(d)) = \omega ( \varphi(c) \tau_t (\varphi(d)) = F_{\varphi(c),\varphi(d)}(t),$$ and similarly at $t+i\beta$. Also, note that pullback is continuous, since $C^*$-algebra isomorphism is compatible with the topology on the set of $\KMS$-states. 

For (ii), if a $\KMS_\beta$ state $\omega$ on $B$ is not extremal, then the GNS-representation $\pi_\omega$ of $\omega$ is not factorial.
As in Prop 3.8 of \cite{CM2}, there exists a positive linear functional, which is dominated by $\omega$, namely $\omega_1 \leq \omega$, and which extends from $B$ to the von Neumann algebra given by the weak closure $\cM_\omega$ of $B$ in the GNS representation. The functional $\omega_1$ is of the 
form $\omega_1 (b)=\omega(hb)$ for some positive element $h$ in the center of the von Neumann algebra $\cM_\omega$. Consider then the pull back $$\varphi^*(\omega)(a)=\omega(\varphi(a))$$ and
$$\varphi^*(\omega_1)(a)=\omega_1(\varphi(a)) =\omega(h \varphi(a))$$ for $a \in A$. The continuous linear 
functional $\varphi^*(\omega_1)$ has norm $\| \varphi^*(\omega_1)\| \leq 1$. In fact, since we are dealing with unital algebras, $$\| \varphi^*(\omega_1)\| =\varphi^*(\omega_1)(1)=\omega(h).$$
The linear functional $\omega_2(b)=\omega((1-h)b)$ also satisfies the positivity property
$\omega_2(b^* b)\geq 0$, since $\omega_1 \leq \omega$. The decomposition 
$$\varphi^*(\omega) = \lambda \eta_1 + (1-\lambda) \eta_2,$$
with $\lambda = \omega(h)$, $$\eta_1=\varphi^*(\omega_1)/\omega(h)\mbox{ and }\eta_2= \varphi^*(\omega_2)/\omega(1-h)$$ shows that the state $\varphi^*(\omega)$ is
not extremal. Notice that $\eta_1$ and $\eta_2$ are both $\KMS$ states. To see this, it suffices
to check that the state $\omega_1(b)/\omega(h)$ is $\KMS$. In fact, one has for all analytic elements $a,b \in B$:
$$\omega_1(ab) = \omega(hab) =
\omega(a h b) = \omega(hb\tau_{i\beta}(a)).$$ 
\end{proof}

\begin{df} An \emph{automorphism} of a QSM-system $(A,\sigma)$ is an isomorphism to itself. The group of such automorphisms is denoted by $\Aut(A,\sigma)$. 

An \emph{endomorphism} of a QSM-system $(A,\sigma)$ is a $\ast$-homomorphism $A \rightarrow A$ that commutes with $\sigma_t$ for all $t$. We denote them by $\End(A,\sigma)$. 

An \emph{inner endomorphism} is defined by $a \mapsto uau^*$ for some isometry $u \in A$ which is an eigenvector of the time evolution, i.e., $u^*u=1$ and there exists an eigenvalue $\lambda$ such that $\sigma_t(u)=\lambda^{it} u$ for all $t$. We denote them by $\mathrm{Inn}(A,\sigma)$. (Inner endomorphisms act trivially on $\KMS$-states, cf.\ \cite{CM}, Ch.\ 3, Section 2.3.)

If $A^\dagger \subset A$ is a dagger-subalgebra preserved by the time
evolution, we denote by $\mathrm{Inn}^{\dagger}(A,\sigma)$ the set of \emph{dagger inner endomorphisms}: the inner endomorphisms
of $(A,\sigma)$ defined by isometries in $A^{\dagger}$ that are eigenvectors of the time evolution.
\end{df}

\section{A QSM-system for number fields} \label{s2}

Bost and Connes (\cite{BC}) introduced a QSM-system for the field of rational numbers, and \cite{CMR},
\cite{CMR2} did so for imaginary quadratic fields. More general QSM-systems associated to arbitrary number fields were constructed by Ha and Paugam in \cite{HP} as a special case of their more
general class of systems for Shimura varieties, which in turn generalize the $\GL(2)$-system
of \cite{CM2}. We recall here briefly the construction of the systems for number fields in an
equivalent formulation (cf.\ also \cite{LLN}). 

\begin{se}We denote by $J^+_{\bK}$ the semigroup of integral ideals, with the norm function
$$N \, : \, J_{\bK}^+ \rightarrow \Z \, : \, \fn \mapsto N(\fn)=N^{\bK}_{\Q}(\fn)=N_{\bK}(\fn).$$ Denote by $G^{\ab}_{\bK}$ the Galois group of the
maximal abelian extension of $\bK$. The Artin reciprocity map is denote by 
$$ \vartheta_{\bK} \, : \, \A_{\bK}^* \rightarrow G_{\bK}^{\ab}.$$
By abuse of notation, we will also write $\vartheta_{\bK}(\fn)$ for the image under this map of an ideal $\fn$, which is seen as an idele by choosing a non-canonical section $s$ of 
$$\xymatrix@R=0pt{ \A_{\bK,f}^* \ar@{->>}[r] & J_{\bK} \ar@/^1.5pc/[l]_{s} &  :  & (x_{\p})_{\p} \mapsto \displaystyle\prod\limits_{\p \mbox{\footnotesize{ finite }}}  \p^{v_{\p}(x_{\p})} }$$ 

The abuse lies in the fact that the image depends on this choice of section (thus, up to a unit in the finite ideles), but it is canonically defined in (every quotient of) the Galois group $G_{\bK,\fn}^{\ab}$ of the maximal abelian extension unramified at prime divisors of $\fn$: on every finite quotient of this, it is the ``Frobenius element'' of $\fn$. The notation $\vartheta_{\bK}(\fn)$ will only occur in situations where this ambiguity plays no role, for example, we evaluate characters $\chi$ on $\vartheta_{\bK}(\fn)$ only if the conductor $\f_\chi$ of $\chi$ is coprime to $\fn$ (so $\chi$ factors over $G_{\bK,\fn}^{\ab}$). If $\fn=\p$ is a prime ideal with a chosen uniformizer $\pi_{\p}$ then we get a diagram 

$$\xymatrix@R=0pt{ J_{\bK}^+ \ar@{->}[r]^{s} \ar@/^1.5pc/[rrr] & \A_{\bK}^* \ar@{->>}[r]^{\vartheta_{\bK}}   & G_{\bK}^{\ab} \ar@{->>}[r] & G_{\bK,\p}^{\ab}  \\ 
\p \ar@{->}[r]  & (1,\dots,1,\pi_{\p},1,\dots,1)\ar@{->}[rr] & & \vartheta_{\bK}(\p)     }$$
in which the arrow $\vartheta_{\bK} \circ s$ depends on $s$, but the curved arrow doesn't depend on $s$.

We consider the fibered product $$ X_{\bK}:=G^{\ab}_{\bK}\times_{\hat\cO_{\bK}^*} \hat\cO_{\bK},$$
(where $\hat\cO_{\bK}$ is the ring of finite integral adeles), where the balancing is defined for $\gamma \in G_{\bK}^{\ab}$ and $\rho \in \hat{\cO}_{\bK}$ by the equivalence $$(\gamma,\rho) \sim (\vartheta_{\bK}({u}^{-1}) \cdot \gamma, u \rho) \mbox{ for all } u \in \hat\cO_{\bK}^*.$$
\end{se}

\begin{df} \label{dfsystem}
The \emph{QSM-system $(A_{\bK}, \sigma_{\bK})$ associated to a number field $\bK$} is defined by
\begin{equation}\label{QSMnf}
A_{\bK}:=C(X_{\bK}) \rtimes J^+_{\bK} =  C(G^{\ab}_{\bK}\times_{\hat\cO_{\bK}^*} \hat\cO_{\bK})\rtimes J^+_{\bK},
\end{equation}
where the crossed product structure is given by $\fn \in J_{\bK}^+$ acting on $f \in C(X_{\bK})$ as
$$ \rho_{\fn}(f)(\gamma,\rho)=f(\vartheta_{\bK}(\fn) \gamma, s(\fn)^{-1} \rho)e_{\fn},  $$
with $e_{\fn}=\mu_{\fn} \mu_{\fn}^*$ the projector onto the space of $[(\gamma,\rho)]$ where $s(\fn)^{-1}\rho \in \hat\cO_{\bK}$. Here $\mu_{\fn}$ is the isometry that implements the action of $J_{\bK}^+$. 

Note that, because of the balancing over the finite idelic units $\hat \cO^*_{\bK}$, the dependence of $\vartheta_{\bK}(\fn)$ on $s$ is again of no influence. By further slight abuse of notation, we will leave out the section $s$ from the notation, and write the action as $f \mapsto f(\vartheta_{\bK}(\fn) \gamma, \fn^{-1} \rho)e_{\fn}$.

Of further use to us will be the partial inverse to this action defined by 
$$ \sigma_{\fn}(f)(x) =f(\fn \ast x)$$
where we have defined the action $\fn \ast x$ of an ideal $\fn \in J_{\bK}^+$ on an element $x \in X_{\bK}$ as 
$$ \fn \ast [(\gamma,\rho)]=[(\vartheta_{\bK}(\fn)^{-1} \gamma, \fn \rho)].$$
Then indeed, 
$$ \mu_{\fn} \mu^*_{\fn} =e_{\fn}; \ \mu_{\fn}^* \mu_{\fn} = 1; \ \rho_{\fn}(f) = \mu_{\fn} f \mu_{\fn}^*; $$ $$ \sigma_{\fn}(f)=\mu_{\fn}^* f \mu_{\fn};\ \sigma_{\fn}(\rho_{\fn}(f)) = f; \ \rho_{\fn}(\sigma_{\fn}(f))=fe_{\fn}. $$

The dagger subalgebra $A_{\bK}^{\dagger}$ is the algebraic crossed product 
generated by functions $f\in C(X_{\bK})$ and isometries $\mu_{\fn}$ with
the relations 
\begin{equation}\label{relsmunuialg}
 \mu_{\fn} f = \rho_{\fn}(f) \mu_{\fn}, \ \ \ \   f \mu_{\fn} = \mu_{\fn} \sigma_{\fn}(f) e_{\fn}, 
 \end{equation}
where $\rho_{\fn}$ and $\sigma_{\fn}$ are as in Section \ref{dfsystem}. 
This is not an involutive subalgebra because it does not contain the adjoints
$\mu_{\fn}^*$, but $A_{\bK}$ is the $C^*$-algebra generated by $A_{\bK}^{\dagger}$.

Finally, the time evolution
is given by
\begin{equation}\label{sigmaK}
\left\{ \begin{array}{ll} \sigma_{\bK,t}(f) =f, & \forall f \in C(G^{\ab}_{\bK}\times_{\hat\cO_{\bK}^*} \hat\cO_{\bK}); \\ \sigma_{\bK,t}(\mu_{\fn}) = N(\fn)^{it}\, \mu_{\fn}, & \forall \fn \in  J^+_{\bK}. \end{array} \right.
\end{equation}
where $\mu_{\fn}$ are the isometries that implement the semigroup action of $J^+_{\bK}$.
The time evolution preserves the dagger subalgebra $A_{\bK}^{\dagger}$.
\end{df}

\section{Hilbert space representation, partition function, KMS-states}

\begin{se}\label{kmschar} A complete classification of the $\KMS$ states for the systems $(A_{\bK},\sigma_{\bK})$ was obtained in \cite{LLN}, Thm.\ 2.1.
In particular, in the low temperature range $\beta > 1$, the extremal $\KMS_\beta$ states are parameterized 
by elements $\gamma \in G_{\bK}^{\ab}$, and are in Gibbs form, given by normalized $L$-series
\begin{equation}\label{KMSlow}
\omega_{\beta,\gamma} (f) = \frac{1}{\zeta_{\bK}(\beta)} \sum_{\fn \in J_{\bK}^+} \frac{f(\fn \ast \gamma)}{ 
N(\fn)^{\beta}}.
\end{equation}
Let $\chi$ denote a character of $G_{\bK}^{\ab}$ (extended as usual by $0$ on ideals not coprime to its conductor $\f_\chi$). We define a function $f_\chi \in C(X_{\bK})$ by 
\begin{equation} \label{fchi1} f_\chi(\gamma,\rho):= \left\{ \begin{array}{ll} \chi^{-1}(\gamma \vartheta_{\bK}(\rho')) & \mbox{ if } \forall v \mid \f_\chi, \rho_v \in \hat\cO_{\bK,v}^*\\ 0 & \mbox{ otherwise, } \end{array} \right. \end{equation}
with $\rho' \in \hat\cO_{\bK}^*$ any invertible integral idele such that $\rho'_v=\rho_v$ for all $v \mid \f_\chi$ (the value is independent of this choice). 
Then from the definition we get \begin{equation} \label{fchi2} f_{\chi}(\fn \ast \gamma) = \left\{ \begin{array}{ll} \chi(\vartheta_{\bK}(\fn)) \chi^{-1}(\gamma) & \mbox{ if } (\fn;\f_\chi)=1, \\ 0 & \mbox{ otherwise, } \end{array} \right.\end{equation} 
so that
\begin{equation}\label{KMSlowa}
\omega_{\beta,\gamma} (f_\chi) = \frac{1}{\zeta_{\bK}(\beta)\chi(\gamma)} \cdot L_{\bK}(\chi,\beta), \end{equation}
is up to normalization the usual $L$-series of $\chi$ (which is defined using the convention to sum only over ideals coprime to the conductor of the $\chi$). 
\end{se}

\begin{se}\label{hil}

Associated to any element 
$\gamma \in G_{\bK}^{\ab}$ is a natural representation $\pi_{\gamma}$ of the algebra $A_{\bK}$ on the Hilbert space $\ell^2(J_{\bK}^+)$. Namely, let $\varepsilon_{\fm}$ denote the canonical basis of $\ell^2(J_{\bK}^+)$. 
Then the action on $\ell^2(J_{\bK}^+)$ of an element
$ f_{\fn} \mu_{\fn} \in A_{\bK}$ with $\fn \in J_{\bK}^+$ and $f_{\fn} \in C(X_{\bK})$ is given by
$$ \pi_\gamma(f_{\fn} \mu_{\fn}) \,\, \varepsilon_{\fm} = f_{\fn}(\fn \fm \ast \gamma)
\, \varepsilon_{\fn \fm}. $$
In this picture, the time evolution is implemented (in the sense of formula (\ref{imp})) by a Hamiltonian 
\begin{equation}\label{HamK}
H_{\sigma_{\bK}} \varepsilon_{\fn} = \log N(\fn) \,\, \varepsilon_{\fn}.
\end{equation}
\end{se} 

\begin{se} 
In this representation,
$$ \tr( \pi_\gamma(f) e^{-\beta H_{\sigma_{\bK}}}) = \sum_{\fn \in J_{\bK}^+} \frac{f(\fn \ast \gamma)}{
N(\fn)^{\beta}}. $$
Setting $f=1$, the Dedekind zeta function $$\zeta_{\bK}(\beta) =\sum\limits_{\fn\in J_{\bK}^+} N(\fn)^{-\beta}$$ appears as the partition function $$\zeta_{\bK}(\beta) =\tr( e^{-\beta H_{\sigma_{\bK}}})$$ of the system (convergent for $\beta>1$). 
\end{se}

\begin{remark}[Formulation in terms of $\bK$-lattices] As shown in \cite{CM2} and \cite{CM}, the original Bost--Connes system admits a geometric reformulation
in terms of commensurability classes of one-di\-men\-si\-o\-nal $\Q$-lattices, which in Section 3 of \cite{LLN} was generalized to number fields. More specifically,  the moduli space of $\bK$-lattices up to scaling is the abelian part $C(X_{\bK})$ of the algebra (a classical quotient), and the moduli space up to scaling \emph{and} commensurability exhibit the complete algebra (a genuinely noncommutative space). We recall the definitions for convenience.

Denote by $\bK_\infty=\prod_{v|\infty} \hat{\bK}_v$ the product
of the completions at the archimedean places, and by $(\bK_\infty^*)^0$ the connected
component of the identity in $\bK^*_\infty$. An \emph{$1$-dimensional $\bK$-lattice} is a pair $(\Lambda, \phi)$, where $\Lambda\subset \bK_\infty$
is a lattice with $\cO_{\bK}\Lambda = \Lambda$ and $\phi: \bK/\cO_{\bK} \to \bK \Lambda/\Lambda$
is an $\cO_{\bK}$-module homomorphism. The set of one-di\-men\-si\-o\-nal $\bK$-lattices can be identified 
with
\begin{equation}\label{modspaceKlatt}
\cM_{\bK,1} =   \A_{\bK}^* / \bK^* \times_{\hat\cO^*_{\bK}} \hat\cO_{\bK}, 
\end{equation}
as in \cite{CMR} and \cite{ConsM}, cf.\ \cite{LLN} Lemma 3.3.  Two $\bK$-lattices are \emph{commensurable}, denoted by $$(\Lambda_1,\phi_1)\sim (\Lambda_2,\phi_2),$$
if $\bK \Lambda_1=\bK \Lambda_2$ and $\phi_1 =\phi_2$ modulo $\Lambda_1+\Lambda_2$.

The \emph{scaling equivalence} corresponds to identifying one-dimensional
$\bK$-lattices $(\Lambda,\phi)$ and $(k\Lambda, k\psi)$, where $k\in (\bK_{\infty}^*)^0$
and $\psi$ is a pointwise limit 
of elements $r \phi$ with $r \in \cO_{\bK}^* \cap (\bK_\infty^*)^0$. The resulting convolution
algebra corresponds to the action of $\A_{\bK,f}^*/\hat \cO^*_{\bK} \simeq J_{\bK}$ on the \emph{moduli space of one-dimensional $\bK$-lattices up to scaling}
$$ \bar{\cM}_{\bK,1} = \A_{\bK}^*/\overline{\bK^* (\bK_\infty^*)^0} \times_{\hat\cO^*_{\bK}} \hat\cO_{\bK} \simeq 
G^{ab}_{\bK} \times_{\hat\cO^*_{\bK}} \hat\cO_{\bK}. $$

The algebra $A_{\bK}$ can be interpreted as the quotient of the groupoid of the commensurability relation by the scaling action. The Hilbert space construction can be fit into the general framework of groupoid algebra representations. 

In the lattice picture, the low temperature KMS states are parameterized by the \emph{invertible} one-dimensional
$\bK$-lattices, namely those for which the $\cO_{\bK}$-module homomorphism $\varphi$ is actually an isomorphism, see \cite{CM}, \cite{CMR}, \cite{LLN}, and Chapter 3 of \cite{CM2}.
\end{remark}

\section{Hamiltonians and arithmetic equivalence}

We first show that the existence of an isomorphism of the quantum statistical mechanical
systems implies arithmetic equivalence; this is basically because the zeta functions of $\bK$ and $\bL$ are the partition functions of the respective systems. Some care has to be taken since the systems are not represented on the same Hilbert space. 

\begin{prop}\label{isotoareq}
Let $\varphi: (A_{\bK},\sigma_{\bK}) \to (A_{\bL},\sigma_{\bL})$ be an isomorphism of QSM-systems of number fields $\bK$ and $\bL$. Then $\bK$ and $\bL$ are arithmetically equivalent, i.e., they have the same Dedekind zeta function.
\end{prop}

\begin{proof}  The isomorphism $\varphi: (A_{\bK},\sigma_{\bK}) \to (A_{\bL}, \sigma_{\bL})$ 
induces an identification of the sets of extremal $\KMS$-states of the two systems, via pullback
$\varphi^*: \KMS_\beta(\bL) \to \KMS_\beta(\bK)$. 

Consider the GNS representations associated to regular low temperature 
$\KMS$ states $\omega=\omega_\beta$ and $\varphi^*(\omega)$.  We denote the
respective Hilbert spaces by ${\mathcal H}_\omega$ and 
${\mathcal H}_{\varphi^*\omega}$. As in Lemma 4.3 of \cite{CCM}, we observe
that the factor ${\mathcal M}_\omega$ obtained as the weak closure of $A_{\bL}$ 
in the GNS representation is of type I$_\infty$, since we are only considering the
low temperature KMS states that are of Gibbs form. Thus, the space 
${\mathcal H}_\omega$ decomposes as 
$${\mathcal H}_\omega = {\mathcal H}(\omega)\otimes {\mathcal H}',$$
with an irreducible representation $\pi_\omega$ of $A_{\bL}$ on ${\mathcal H}(\omega)$
and $${\mathcal M}_\omega =\{ T \otimes 1\,|\, T\in \mathcal{B}( {\mathcal H}(\omega)) \}$$
($\mathcal{B}$ indicates the set of bounded operators). Moreover, we have
$$ \langle (T\otimes 1) 1_\omega, 1_\omega\rangle = {\rm Tr}(T \rho) $$
for a density matrix $\rho$ (positive, of trace class, of unit trace). 

We know that the low temperature extremal KMS states for the system 
$(A_{\bL}, \sigma_{\bL})$ are of Gibbs form and given by the explicit expression
in equation (\ref{KMSlow})
for some $\gamma \in G_{\bL}^{\ab}$; and similarly for the system $(A_{\bK},\sigma_{\bK})$. Thus, we can identify
${\mathcal H}(\omega)$ with $\ell^2(J_{\bL}^+)$ and the 
density $\rho$ correspondingly with 
$$\rho=e^{-\beta H_{\sigma_{\bL}}}/{\rm Tr}(e^{-\beta H_{\sigma_{\bL}}});$$ this is the representation considered in Section \ref{hil}. As in
Lemma 4.3 of \cite{CCM}, the evolution group $e^{it H_\omega}$ generated by the
Hamiltonian $H_\omega$ that implements
the time evolution $\sigma_{\bL}$ in the GNS representation on ${\mathcal H}_\omega$ 
agrees with $e^{it H_{\sigma_{\bL}}}$ on the factor ${\mathcal M}_\omega$. We find
$$ e^{it H_\omega} \pi_\omega(f) e^{-it H_\omega} = \pi_\omega(\sigma_{\bL}(f)) =
e^{it H_{\sigma_{\bL}}} \pi_\omega(f)  e^{-it H_{\sigma_{\bL}}}. $$
As observed in \S 4.2 of \cite{CCM}, this gives us that the Hamiltonians differ by a constant, 
$$H_\omega = H_{\sigma_{\bL}} + \log \lambda_1,$$ for some $\lambda_1\in \R^*_+$. 
The argument for the GNS representation
for $\pi_{\varphi^*(\omega)}$ is similar and it gives an identification
of the Hamiltonians $$H_{\varphi^*(\omega)} = H_{\sigma_{\bK}} + \log \lambda_2$$ for
some constant $\lambda_2\in \R^*_+$. 

The algebra isomorphism $\varphi$ induces a unitary equivalence $\Phi$ 
of the Hilbert spaces of the GNS representations of the corresponding states,
and the Hamiltonians that implement the time evolution in these representations
are therefore related by $$H_{\varphi^*(\omega)} = \Phi H_{\omega} \Phi^*.$$
In particular the Hamiltonians $H_{\varphi^*(\omega)}$ and $H_\omega$ then
have the same spectrum. 

Thus, we know from the discussion above that $$H_{\bK} = \Phi H_{\bL} \Phi^* + \log \lambda,$$
for a unitary operator $\Phi$ and a $\lambda\in \R^*_+$.
This gives at the level of zeta functions
\begin{equation} \label{grom} \zeta_{\bL} (\beta) = \lambda^{-\beta} \zeta_{\bK}(\beta). \end{equation}
Now consider the left hand side and right hand side as classical Dirichlet series of the form $$\sum\limits_{n \geq 1} \frac{a_n}{n^{\beta}} \mbox{ \ and \ }\sum\limits_{n \geq 1} \frac{b_n}{(\lambda n)^{\beta}},$$ respectively. Observe that $$a_1=b_1=1.$$ Taking the limit as $\beta \rightarrow + \infty$ in \eqref{grom}, we find 
$$ a_1 = \lim_{\beta \rightarrow +\infty} b_1 \lambda^{-\beta}, $$ from which we conclude that $\lambda=1$. 
Thus, we obtain
$\zeta_{\bK}(\beta) = \zeta_{\bL}(\beta)$, which gives arithmetic equivalence
of the number fields.
\end{proof}

By expanding the zeta functions as Euler products, we deduce

\begin{cor}
If the QSM-systems $(A_{\bK},\sigma_{\bK})$ and $(A_{\bL},\sigma_{\bL})$ of two number fields $\bK$ and $\bL$ are isomorphic, then there is a bijection of the primes $\p$ of $\bK$ above $p$ and the primes $\q$ of $\bL$ above $p$ that preserves the inertia degree: $f(\p|{\bK})=f(\q|{\bL})$. \qed
\end{cor}

Using some other known consequences of arithmetical equivalence, we get the following (\cite{Perlis1}, Theorem 1): 

\begin{cor}
If the QSM-systems $(A_{\bK},\sigma_{\bK})$ and $(A_{\bL},\sigma_{\bL})$ of two number fields $\bK$ and $\bL$ are isomorphic, then the number fields have the same degree over $\Q$, the same discriminant, normal closure, isomorphic unit groups, and the same number of real and complex embeddings. \qed
\end{cor}

However, it does not follow from arithmetical equivalence that $\bK$ and $\bL$ have the same class group (or even class number), cf.\ \cite{PdS}. 

\section{Layers of the QSM-system} 

\begin{se}  \label{hairy} The group $G_{\bK}^{\ab}$ has quotient groups $G_{\bK,\fn}^{\ab}$ defined as the Galois group of the maximal abelian extension of $\bK$ which is unramified at primes dividing $\fn$.  This structure is also reflected in the algebra of the QSM-system, cf.\ also \cite{LLN}, proof of Thm.\ 2.1, or section 3 of \cite{CMR} (including a description in terms of $\bK$-lattices). 

Let $\mu_{\bK}$ denote the measure on $$X_{\bK}=G_{\bK}^{\ab} \times_{\hat\cO^*_{\bK}}\hat\cO_{\bK}$$ given as the products of normalized Haar measures on $G_{\bK}^{\ab}$ and on every factor $\hat{\cO}_{\bK,\p}$ of $\hat\cO_{\bK}$ (so that $\hat\cO_{\bK,\p}^*$ has measure $1-1/N_{\bK}(\p)$). 
Fix an ideal $\fn$ and consider the space
$$ X_{\bK,\fn} := G_{\bK}^{\ab} \times_{\hat \cO_{\bK}^*} \hat \cO_{\bK,\fn}, $$
where $\hat \cO_{\bK,\fn} = \prod_{\p \mid \fn} \hat \cO_{\bK,\p}.$
Then $$ X_{\bK} = \lim_{{\longrightarrow}\atop{\fn}} X_{\bK,\fn}. $$
Let $J_{\bK,\fn}^+$ denote the subsemigroup of $J_{\bK}^+$ generated by the prime ideals dividing $\fn$. Consider the subspace $$X_{\bK,\fn}^*:=G_{\bK}^{\ab} \times_{\hat \cO_{\bK}^*} \hat \cO^*_{\bK,\fn}$$ of $X_{\bK,\fn}$. It is isomorphic as a topological group to  \begin{equation} \label{ster} X_{\bK,\fn}^* \cong G_{\bK}^{\ab}/ \vartheta_{\bK}(\hat \cO_{\bK,\fn}^*) = G_{\bK,\fn}^{\ab},\end{equation} 
the Galois group of the maximal abelian extension of $\bK$ that is unramified at the primes dividing $\fn$.  

We can decompose 
$$ X_{\bK,\fn} = X^1_{\bK,\fn} \coprod X_{\bK,\fn}^2$$ with $$ X^1_{\bK,\fn}:=\coprod_{\fm' \in J_{\bK,\fn}^+} \fm'\ast X^*_{\bK,\fn}\quad \mbox{ and }\quad X^2_{\bK,\fn}:=\bigcup_{\p \mid \fn} Y_{\bK, \p}, $$
where $$Y_{\bK,\p} = \{ (\gamma, \rho) \in X_{\bK,\fn} \, : \, \rho_{\p}=0 \} . $$
The decomposition of $X^1_{\bK,\fn}$ is into disjoint subsets, because $[(\gamma,\rho)] \in \fm' \ast X_{\bK,\fn}^*$ precisely if $\rho$ is exactly ``divisible'' by $\fm'$.  

We observe that by Equation (\ref{ster}), we have a homeomorphism \begin{equation} \label{homx1} X^1_{\bK,\fn} \cong \coprod_{\fm' \in J_{\bK,\fn}^+} G_{\bK,\fn}^{\ab}. \end{equation}  
Now by Fourier analysis, the characters of $G_{\bK,\fn}^{\ab}$ (so the characters of $G_{\bK}^{\ab}$ whose conductor is coprime to $\fn$) are dense in the algebra of functions on $G_{\bK,\fn}^{\ab}$. The algebra of continuous functions on the coproduct $C(\coprod_{\fm' \in J_{\bK,\fn}^+} G_{\bK,\fn}^{\ab})$ is then generated by linear combinations of such characters with support in just one of the components. By pulling this back via the homeomorphism in \eqref{homx1}, we find a set of generators for the algebra of continuous functions on $X^1_{\bK,\fn}$: 
\end{se}

\begin{df} Write an element  $x \in X^1_{\bK,\fn} $ as $x=\fm' \ast [(\gamma,\rho)],$ for some $\rho \in \hat\cO_{\bK,\fn}^*$ (so it is in the $\fm'$-component of the decomposition $X^1_{\bK,\fn} = \coprod_{\fm' \in J_{\bK,\fn}^+} \fm'\ast X^*_{\bK,\fn}$). Let $\chi$ denote character of $G_{\bK}^{\ab}$ whose conductor is coprime to $\fn$, and let $\fm \in J_{\bK,\fn}^+$. Then we define the function 
 $$ f_{\chi,\fm} \colon \fm' \ast [(\gamma,\rho)] \mapsto \delta_{\fm,\fm'} \chi(\vartheta_{\bK}(\fm^{-1}) \gamma), $$
 where $\delta_{\fm,\fm'}$ is the Kronecker delta. 
This is the pullback by the homeomorphism in \eqref{homx1} of the function which is the character $\chi$ precisely in the $\fm$-component of the space. 
\end{df}

The above results imply that these functions generate the algebra $C(X_{\bK,\fn}^1)$. We can now prove:

\begin{lem} \label{gengen} The algebra of functions $C(X_{\bK,\fn})$ is generated by the functions $f_{\chi,\fm}$  in $ C(X^1_{\bK,\fn} ),$ for all   $\chi \in \widehat G_{\bK,\fn}^{\ab}$ and ideals $\fm \in J_{\bK,\fn}^+$.
\end{lem}

\begin{proof} Observe that $X_{\bK,\fn}^2$ is a set of $\mu_{\bK}$-measure zero.
By total disconnectedness, the algebra $C(X_{\bK,\fn})$ is generated by the characteristic functions of clopen sets.  We claim: 

\begin{lem} \label{regular}
The space $X_{\bK,\fn}$ has no non-empty open sets of $\mu_{\bK}$-measure zero. 
\end{lem}

\begin{proof}[Proof of Lemma \ref{regular}] A $\p$-adic ring of integers $\hat \cO_{\p}$ does not have non-empty open sets $U$ of measure zero, since $U$ contains a ball of sufficiently small radius around any point in it, and this will have Haar measure the $\p$-adic absolute value of the radius; the same argument applies to $G_{\bK}^{\ab}$, by considering it as the idele class group modulo connected component of the identity and using the idele norm. 
\end{proof}
 
It follows that $X_{\bK,\fn}^1$ is dense in $X_{\bK,\fn}$, as the complement
cannot contain any open set. It therefore suffices to give generators for $C(X^1_{\bK,\fn} )$, which we have already done. 
\end{proof}

\begin{cor} \label{dude}
The set $$ X_{\bK}^1 := \bigcup_{\fm \in J^+_{\bK}} \fm \ast (G_{\bK}^{\ab} \times_{\hat \cO_{\bK}^*} \hat \cO_{\bK}^\ast) = \{ [(\gamma,\rho)] \colon \rho_{\p} \neq 0 \ \forall \p \} $$ is dense in $X_{\bK}$. 
\end{cor}

\begin{proof}
It follows from the above proof that  $$ \bigcup_{\fm \in J^+_{\bK,\fn}} \fm \ast (G_{\bK}^{\ab} \times_{\hat \cO_{\bK}^*} \hat \cO_{\bK,\fn}^\ast)$$ is dense in $X_{\bK,\fn}$, so by taking the union over all $\fn$, we find the result (recall that the closure of a union contains the union of the closures). \end{proof}

\begin{remark}[$\bK$-lattices] Let $\bar \cM_{\bK,1}$ denote the space of one-dimensional $\bK$-lattices up to scaling; recall that $C(X_{\bK})=C(\bar \cM_{\bK,1})$. The preceding theory organizes this space into an inductive system of the spaces $C(\bar \cM_{\bK,1,\fn})$ of functions that depend on the datum $\phi$ of a $\bK$-lattice $(\Lambda,\phi)$  only through its projection to $\hat\cO_{\bK,\fn}$. 

\end{remark}

\section{Crossed product structure and QSM-isomorphism}

In this section, we deduce from the dagger-isomorphism of the QSM-systems the conjugacy of the corresponding ``dynamical systems'' $(X_{\bK},J_{\bK}^+)$ and $(X_{\bL},J_{\bL}^+)$. There is a large literature on recovering such systems from (non-involutive) operator algebras, starting with Arveson-Josephson. We refer to \cite{Davidson} for a recent overview and theorems with minimal conditions, leading to ``piecewise conjugacy''. Here, we will present as simple as possible a proof for our case, where we can exploit our assumption that the \emph{algebraically generated} dagger-subalgebra is preserved, as well as the ergodicity of the action and some strong density assumptions on fixed point sets. 

\begin{notation}
Fix a rational prime number $p$ and a positive integer $f$. Let $J_{\bK,p^f}^+$ denote the sub-semigroup of $J_{\bK}^+$ generated by the primes $\p_1^{\bK}=\p_1,\dots,\p_N^{\bK}=\p_N$ of norm $N_{\bK}(\p_i)=p^f$. Let $A^{\dagger}_{\bK,p^f}$ denote the (non-involutive) 
subalgebra of  $A^{\dagger}_{\bK}$ generated algebraically by the functions $C(X_{\bK})$ and the 
isometries $\mu_{\p}$ with $\p=\p_i$ a prime in $J_{\bK,p^f}^+$.

We will use multi-index notation: for $\alpha=(\alpha_1,\dots,\alpha_N) \in \Z_{\geq 0}^N$, we let $\mu_\alpha=\mu_{\p_1}^{\alpha_1} \dots \mu_{\p_N}^{\alpha_N}$, and let $|\alpha|=N$ denote the length of $\alpha$. Similarly, we let $\sigma_{\alpha}=\sigma_{\mu_{\alpha}}$, etc.\ (beware not to confuse the partial inverse $\sigma_\alpha$ with the time evolution $\sigma_t$).   Any element $a \in A_{\bK,p^f}^\dagger$ can be uniquely written in the form $$a=\sum_{\alpha}   \mu_\alpha E_\alpha (a)$$ for ``generalized Fourier coefficients'' $E_{\alpha}(a) \in C(X_{\bK})$. 
\end{notation}

\begin{prop}\label{isoJs}
A dagger-isomorphism of QSM-systems $\varphi: (A_{\bK},\sigma_{\bK}) \isomto (A_{\bL},\sigma_{\bL})$  induces a homeomorphism $\Phi: X_{\bK}\isomto X_{\bL}$ and a norm-preserving semigroup isomorphism
$\varphi : J_{\bK}^+ \isomto J_{\bL}^+ $ \textup{(}viz., such that $N_{\bL}(\varphi(\fn)) = N_{\bK}(\fn)$\textup{)}, satisfying the compatibility condition
$$ \Phi(\fn \ast x) = \varphi(\fn) \ast \Phi(x).$$
\end{prop}

\begin{proof} 

First of all, $\varphi$ maps the $\sigma_t$-eigenspace for eigenvalue $1$ in the dagger subalgebra $A_{\bK}^\dagger$ to that in $A_{\bL}^\dagger$: from the representation through generalized Fourier series, it is easy to see that these eigenspaces consist exactly of the functions $C(X_{\bK})$, respectively $C(X_{\bL})$. Hence $\varphi$ induces an isomorphism between these algebras, and hence a homeomorphism $$\Phi \colon X_{\bK} \rightarrow X_{\bL}.$$  

Now fix a rational prime $p$ and a positive integer $f$. We claim that $\varphi$ induces an isomorphism $$\varphi: A^{\dagger}_{\bK,p^f}\isomto A^{\dagger}_{\bL,p^f}.$$  Indeed, we have by assumption that $\varphi$ maps the dagger subalgebra $A_{\bK}^\dagger$ to $A_{\bL}^\dagger$. Now $A_{\bK,p^f}^\dagger$ is precisely the subalgebra generated by $C(X)$ and the $p^{fit}$-eigensubspace of $\sigma_t$ acting on $A_{\bK}^\dagger$. Since $\varphi$ is compatible with time evolution, it maps the $p^{fit}$-eigenspace of $\sigma_{\bK,t}$ to that of $\sigma_{\bL,t}$, so the claim holds. 

We now interject a topological lemma which will be used in the proof: 

\begin{lem} \label{denseopen}
\mbox{ } 
\begin{enumerate}
\item[\textup{(i)}] Let $x=[(\gamma,\rho)] \in X_{\bK}$ and assume that there exist two distinct ideals $\fm$ and $\fn$ with 
$\fm \ast x = \fn \ast x. $ Then the $\p$-component $x_{\p}$ of $x$ is zero for some $\p$ dividing the least common multiple of $\fm$ and $\fn$. 
\item[\textup{(ii)}] The set $$ X_{\bK}^0:=\{ x \in X_{\bK} \colon \fm \ast x \neq \fn \ast x \mbox{\ for all } \fm \neq \fn \in J_{\bK,p^f}^+ \} $$
contains a dense open set in $X_{\bK}$.
\item[\textup{(iii)}] The set $X_{\bK}^{00}:= X_{\bK}^0 \cap \Phi^{-1}(X_{\bL}^0)$ is dense in $X_{\bK}$.  
\end{enumerate}
\end{lem}

\begin{proof}
The equality  $\fm \ast x = \fn \ast x$ means the existence of an idelic unit with $\vartheta_{\bK}(\fm) = \vartheta_{\bK}(u) \vartheta_{\bK}(\fn)$ and $\rho s(\fm) = u s(\fn) \rho$. Thus, if $\rho$ has non-zero component at all divisors of $\fm$ and $\fn$, then it follows from the second equality that $\fm=\fn$. 

Now consider the set consisting of $x \in X_{\bK}$ such that $x_{\p} \neq 0$ for \emph{all} $\p=\p_1,\dots,\p_N$ in $J_{\bK,p^f}^+$. By the above, it is contained in $X_{\bK}^0$. Also the set is open, as the complement of finitely many closed sets (namely, the ones on which $x_{\p}=0$ for the finitely many primes $\p$ of norm $p^f$). Finally, it is dense, since it contains the set $X_{\bK}^1$ (the subset where \emph{no} component of $\rho$ is zero), of which we have already shown that it is dense in $X_{\bK}$ in Lemma \ref{dude}. 

Since $\Phi$ is a homeomorphism, $\Phi^{-1}(X_{\bL}^0)$ is dense open in $X_{\bK}$, and it suffices to notice that the intersection of dense open sets is dense. \end{proof}

We now show that one can algebraically describe the set of images of $x \in X^0_{\bK}$ under the generators $\p_i$: 

\begin{lem}
Let $\cC$ denote the commutator ideal in $A_{\bK,p^f}^\dagger$ and $\cC^2$ the span of products of elements in $\cC$. For $x_0 \in X_{\bK}$, let $I_{x_0}$ denote the ideal of functions $f \in C(X_{\bK})$ that vanish at $x_0$. 
Then for $x_0, y_0 \in X^0_{\bK}$, we have that $$y_0 \in \{ \p_1 \ast x_0,\dots, \p_N \ast x_0 \}$$ if and only if 
$$M_{x_0,y_0}:= I_{y_0} \cC + \cC I_{x_0} +\cC^2$$ has codimension one as a subvectorspace of $\cC$.
\end{lem}

\begin{proof} 
We claim that 
$$ \cC = \{ a \in A_{\bK,p^f}^\dagger \colon E_0(a)=0 \mbox{ and } E_\alpha(a) \in \cE_\alpha \ \forall \alpha \neq 0 \}, $$
where $\cE_\alpha$ is the $C(X_{\bK})$-ideal generated by 
the ``coboundaries" $$h=f-\sigma_\alpha(f) $$ for some $f \in C(X_{\bK})$. Indeed, this follows from computing commutators $[\mu_\alpha,f] =  \mu_\alpha(f-\sigma_{\alpha}(f))$ for $f \in C(X_{\bK})$. 
Similarly, one finds 
$$ \cC^2 = \{ a \in A_{\bK,p^f}^\dagger \colon E_\alpha(a)=0 \mbox{ for all } |\alpha| \leq 1 \mbox{ and } E_\alpha(a) \in \cE^2_\alpha \ \forall |\alpha|>1 \}, $$
where $\cE^2_\alpha$ is the ideal in $C(X_{\bK})$ generated by products of coboundaries. Since the action of $J_{\bK,p^f}^+$ is continuous, the ideals $\cE_\alpha$ are closed in $C(X_{\bK})$, so $\cE^2_\alpha=\cE_\alpha$. 
Hence the space $\cC/\cC^2$ is isomorphic to $\cC/\cC^2= \bigoplus\limits_{|\alpha|=1} \mu_{\alpha} \cE_{\alpha} $. 

Now $M_{x_0,y_0}=I_{y_0} \cC + \cC I_{x_0} +\cC^2$ is described as 
$$ M_{x_0,y_0} = \{ a \in \cC \colon E_{\alpha}(a) \in \left(\sigma_{\alpha}(I_{y_0}) +I_{x_0} \right)\cE_{\alpha} \ \forall |\alpha|=1 \} $$ 
Fix an index $|\beta|=1$, corresponding to $\p_k$. Since $I_{x_0}$ is a closed maximal ideal in $C(X)$, either $\sigma_{\beta}(I_{y_0}) \subseteq I_{x_0}$, or $\sigma_\beta(I_{y_0}) + I_{x_0} = C(X_{\bK})$. The first case occurs exactly if $y_0 = \p_k \ast x_0$. Also, this case occurs at most for one such $k$, since we assume that $x_0 \in X_{\bK}^0$. Hence either 
$ M_{x_0,y_0} =  \cC$, or there exists a unique $k$ (so a unique corresponding $\beta$) such that $y_0 = \p_k \ast x_0$ and $M_{x_0,y_0} = \{ a \in \cC \colon E_{\beta}(a) \in I_{x_0}\cE_{\beta}\} $, which has codimension $1$ in $\cC$.  
\end{proof} 

Now recall that we know that $\varphi$ is induced from a homeomorphism $\Phi: X_{\bK} \rightarrow X_{\bL}$. Since $\varphi$ is an algebra homomorphism, we find that 
$$ \varphi(M^{\bK}_{x_0,y_0}) = M^{\bL}_{\Phi(x_0),\Phi(y_0)}  $$
(where we use superscript $\bK$ and $\bL$ to refer to the different fields). 
Now suppose that $x \in X^{00}_{\bK}$. Then the sets $\{\p^{\bK}_i \ast x \}_{i=1}^N$ and $\{\p^{\bL}_i \ast \Phi(x)\}_{i=1}^N$ contain $N$ distinct elements, and the above reasoning shows that they are mapped to each other by $\Phi$: 
this gives, for each 
$x\in X_{\bK}^{00}$, a permutation of $\p^{\bL}_i$, and hence a locally constant function $\alpha: X_{\bK}^{00} \times J_{\bK,p^f}^+ \to J_{\bL,p^f}^+$ with 
\begin{equation} \label{dodo} \Phi(\fn \ast x) = \alpha_x(\fn) \ast \Phi(x). \end{equation}  
Since $X_{\bK}^{00}$ is dense in $X_{\bK}$, we can extend $\alpha$ by continuity to $X_{\bK} \times J_{\bK,p^f}^+$, such that identity \eqref{dodo} still holds. 

Gluing back together the algebras $A_{\bK,p^f}^\dagger$ for various $p$ and $f$, we finally find a homeomorphism 
$\Phi \colon X_{\bK} \isomto X_{\bL}$ (which is by construction independent of $p^f$), and a locally constant map
$$ \alpha \colon X_{\bK} \times J_{\bK}^+ \rightarrow J_{\bL}^+ $$
such that $N_{\bL}(\alpha_x(\fn)) = N_{\bK}(\fn)$,  and \eqref{dodo} holds for all $x$ and $\fn \in J_{\bK}^+$ (this is known as \emph{piecewise conjugacy} of the dynamical systems $(X_{\bK},J_{\bK}^+)$ and $(X_{\bL},J_{\bL}^+)$ in the terminology of Davidson and Katsoulis \cite{Davidson}).

We now proceed to showing that $\alpha_x$ is actually constant. For this, consider the level set 
$$ \tilde{X}_{\bK} := \{ x \in X_{\bK} \colon \Phi(x) \in X_{\bL}^1  \mbox{ and } \alpha_x(\fn)=\alpha_1(\fn) \ \forall \fn \in J_{\bK}^+ \}. $$
Observe that we only consider $x$ for which $\Phi(x)$ is in $X_{\bL}^1$, the dense subspace of $X_{\bL}$ in which none of the idele components is zero (cf.\ \ref{dude}).
 
 We claim that the set $\tilde{X}_{\bK}$ is invariant under the action of $J_{\bK}^+$. We will verify that  for all $\fm \in J_{\bK}^+$, we have that $\alpha_x=\alpha_1$ if and only if $\alpha_{\fm \ast x} = \alpha_1$. 
 
 We compute that for $\fn \in J_{\bK}^+$ one has
\begin{eqnarray}  \label{boho} \alpha_{\fm \ast x}(\fn) \ast \Phi(\fm \ast x) &=& \Phi(\fn \ast (\fm \ast x))  = \Phi(\fm \fn \ast x) \\ &=& \alpha_x(\fm \fn) \ast \Phi(x) =  \alpha_x(\fn) \ast (\alpha_x(\fm) \ast \Phi(x)) \nonumber \\ &=& \alpha_x(\fn) \ast \Phi(\fm \ast x). \nonumber\end{eqnarray}
We now claim that if $\Phi(x) \in X_{\bL}^1$, then also $\Phi(\fm \ast x) \in X_{\bL}^1$ for all $\fm \in J_{\bK}^+$; this follows from the compatibility $\Phi(\fm \ast x) = \alpha_x(\fm) \ast \Phi(x)$ and the fact that $\Phi(x)=[(\gamma,\rho)] \in X_{\bL}^1$ if and only if none of the local components $\rho_{\p}$ of $\rho$ is zero, which is preserved under the action of $\alpha_x(\fm)$. Hence in the above formula, $\Phi(\fm \ast x) \in X_{\bL}^1$. 

Now if $y  \in X_{\bL}^1$, then by Lemma \ref{denseopen}, for any ideals $\fm', \fn' \in J_{\bL}^+$, we have an equivalence \begin{equation} \label{mum} \fm' \ast y = \fn' \ast y \ \iff \ \fm' = \fn'. \end{equation} 

Thus, we conclude from \eqref{boho} that we have an equality of ideals $\alpha_{\fm \ast x}(\fn) = \alpha_x(\fn)$ for all $\fn \in J_{\bK}^+$. Hence $\alpha_x=\alpha_1$ if and only if $\alpha_{\fm \ast x} = \alpha_1$, which shows that $\tilde{X}_{\bK}$ is an invariant set for the action of $J_{\bK}^+$ on $X_{\bK}$. 

Now recall from \cite{LLN} (Proof of Theorem 2.1 on p.\ 332) that the action of $J_{\bK}^+$ on $X_{\bK}$ is ergodic for the  measure $\mu_{\bK}$ (cf.\ Section \ref{hairy}). Thus, the invariant set $\tilde{X}_{\bK}$ has measure zero or one. It cannot have measure zero: it contains the element $x=1$, and since $\alpha_x$ is locally constant, it contains an open neighbourhood of $1$, and non-empty open sets in $X_{\bK}$ have strictly positive measure (by Lemma \ref{regular}). We conclude that $\tilde{X}_{\bK}$ is of full measure hence also its superset 
$$ \tilde{X}'_{\bK}= \{ x \in X_{\bK} \colon \alpha_x = \alpha_1 \}$$
is of full measure and closed. Hence the complement is an open set of measure zero, hence empty (Lemma \ref{regular}). We conclude that $\tilde{X}'_{\bK}=X_{\bK}$ and we indeed have $\alpha_x=\alpha_1$ for all $x \in X_{\bK}$. 
\end{proof}

\section{QSM-isomorphism and isomorphism of abelianized Galois groups}

In this section, we prove that QSM-isomorphism implies an isomorphism of abelianized Galois groups. 

\begin{remark} \label{Ulm}
The isomorphism type of the infinite abelian group $G_{\bK}^{\ab}$ is determined by its so-called \emph{Ulm invariants}. For $G_{\bK}^{\ab}$, those were computed abstractly by Kubota (\cite{Kubota}), and Onabe (\cite{Onabe}) computed them explicitly for quadratic imaginary fields. For example, $G_{\Q(i)}^{\ab}$ is never isomorphic to any other group for such a field, but $\Q(\sqrt{-2})$ and $\Q(\sqrt{-3})$ have isomorphic abelianized absolute Galois groups (and they are not isomorphic as fields). 
\end{remark}

\begin{lem}
Consider the projector 
$e_{\bK,\fn}=\mu_{\fn} \mu_{\fn}^*$. Then the range of $e_{\bK,\fn}$ is mapped by $\Phi$ to the range of $e_{\bL,\varphi(\fn)}$: 
$$ \Phi(\mathrm{Range}(e_{\bK,\fn})) = \mathrm{Range}(e_{\bL,\varphi(\fn)}). $$
\end{lem}

\begin{proof}
By definition, we have that $x=[(\gamma,\rho)]$ is in the range of $e_{\bK,\fn}$ if and only if $x=[(\gamma',\fn \rho')]$ for some $\gamma' \in G_{\bK}^{\ab}$ and $\rho' \in \hat \cO_{\bK}$. This is equivalent to 
$$ x = [(\vartheta_{\bK}(\fn)^{-1} \gamma'',\fn \rho')] = \fn \ast x'$$ for some $x'=[(\gamma'',\rho')] \in X_{\bK}$. 
If we now apply $\Phi$, we get that the statement is equivalent to 
$$ \Phi(x) = \Phi(\fn \ast x') = \varphi(\fn) \ast \Phi(x')$$
for some $\Phi(x') \in X_{\bL}$ --- here, we have used Proposition \ref{isoJs}. The latter statement is equivalent to 
$\Phi(x)$ belonging to the range of $e_{\bL,\varphi(\fn)}$. \end{proof}

\begin{prop} \label{gab}
An isomorphism $\varphi$ of QSM-systems $(A_{\bK},\sigma_{\bK})$ and $(A_{\bL},\sigma_{\bL})$ induces a topological group isomorphism 
$$ \tilde{\Phi} := \Phi\cdot \Phi(1)^{-1} \, : \, G_{\bK}^{\ab} \isomto G_{\bL}^{\ab}. $$
\end{prop}

\begin{proof}

Fix an ideal $\fm \in J_{\bK}^+$, and consider the subspace of $X_{\bK}$ given by 
$$ V_{\bK,\fm}:=\bigcap_{(\fm,\fn)=1} \mathrm{Range}(e_{\bK,\fn}) = G_{\bK}^{\ab} \times_{\hat \cO_{\bK}^*} \{(0,\dots,0,\hat\cO_{\bK,\fm},0,\dots,0)\}, $$
with $\hat\cO_{\bK,\fm} = \prod_{\p \mid \fm} \hat\cO_{\bK,\p}$. 
This is mapped by $\Phi$ to 
\begin{eqnarray*} \Phi(V_{\bK,\fm}) &=& \bigcap_{(\fm,\fn)=1} \Phi(\mathrm{Range}(e_{\bK,\fn})) = \bigcap_{(\varphi(\fm),\varphi(\fn))=1} \mathrm{Range}(e_{\bL,\varphi(\fn)})\\ &=&   G_{\bL}^{\ab} \times_{\hat \cO_{\bL}^*} \{(0,\dots,0,\hat\cO_{\bL,\varphi(\fm)},0,\dots,0)\} = V_{\bL,\varphi(\fm)}. \end{eqnarray*}
Now define $1_{\fm}$ to be the integral adele which is $1$ at the prime divisors of $\fm$ and zero elsewhere, and consider the subgroup 
$$H_{\bK,\fm}:=G_{\bK}^{\ab} \times_{\hat \cO_{\bK}^*} \{1_{\fm}\} \subseteq X_{\bK}.$$
By the above, $\Phi(H_{\bK,\fm})$ is a subset of $V_{\bL,\varphi(\fm)}$. 
 
The group $H_{\bK,\fm}$ consists of classes $[(\gamma,1_{\fm})] \sim [(\gamma',1_{\fm})] \iff \exists u \in \hat\cO_{\bK}^*$ with $\gamma'=\vartheta_{\bK}(u)^{-1} \gamma$ and $1_{\fm}=u1_{\fm}$. This last equation means that $u_{\q}=1$ at divisors $\q$ of $\fm$ with no further restrictions, i.e.,  $u \in \prod_{\q \nmid \fm} \hat\cO_{\q}^*$, so that by class field theory
$$ H_{\bK,\fm} \cong G_{\bK}^{\ab} / \vartheta_{\bK}\left(\prod_{\q \nmid \fm} \hat\cO_{\q}^*\right) \cong \mathring{G}_{\bK,\fm}^{\ab}, $$
where $\mathring{G}_{\bK,\fm}^{\ab}$ is the Galois group of the maximal abelian extension of $\bK$ that is unramified \emph{outside} prime divisors of $\fm$. Class field theory implies that $\mathring{G}_{\bK,\fm}^{\ab}$ has a dense subgroup generated by $\vartheta_{\bK}(\fn)$ for $\fn$ running through the ideals $\fn$ that are coprime to $\fm$. Said differently, $H_{\bK,\fm}$ is generated by $\gamma_{\fn}:=[(\vartheta_{\bK}(\fn)^{-1},1_{\fm})]$ for $\fn$ running through the ideals coprime to $\fm$. Write  $\one_{\fm}=[(1,1_{\fm})]$, and $\Phi(\one_{\fm}) = [(x_{\fm},y_{\fm})]$. Since $\fm$ and $\fn$ are coprime, we have $[(\vartheta_{\bK}(\fn)^{-1},1_{\fm})] = [(\vartheta_{\bK}(\fn)^{-1},\fn 1_{\fm})]$, and hence we can write  $\gamma_{\fn} = \fn \ast \one_{\fm}$.

Now for two ideals $\fn_1$ and $\fn_2$ coprime to $\fm$, we can perform the following computation: \begin{eqnarray*}
\Phi(\one_{\fm}) \cdot  \Phi(\gamma_{\fn_1} \cdot \gamma_{\fn_2}) &=& \Phi(\one_{\fm}) \cdot \left( \varphi(\fn_1) \varphi(\fn_2) \ast \Phi(\one_{\fm}) \right) \\ &=& [(\vartheta_{\bL}(\varphi(\fn_1) \varphi(\fn_2))^{-1} x_{\fm}^2, \varphi(\fn_1) \varphi(\fn_2) y_{\fm}^2)] \\ &=& 
 [(\vartheta_{\bL}(\varphi(\fn_1))^{-1} x_{\fm}, \varphi(\fn_1) y_{\fm})] \cdot  [(\vartheta_{\bL}(\varphi(\fn_2))^{-1} x_{\fm}, \varphi(\fn_2) y_{\fm})] \\ &=&  \left( \varphi(\fn_1) \ast \Phi(\one_{\fm}) \right) \cdot \left( \varphi(\fn_2) \ast \Phi(\one_{\fm}) \right)  \\ &=& \Phi(\gamma_{\fn_1}) \cdot \Phi(\gamma_{\fn_2}). 
\end{eqnarray*}
By density, we find that for all $\gamma_1, \gamma_2 \in H_{\bK,\fm}$, we have 
$$ \Phi(\one_{\fm}) \Phi(\gamma_{1} \gamma_{2})  = \Phi(\gamma_{1}) \Phi(\gamma_{2}). $$

We now consider the image $\Phi(H_{\bK,\fm})$. Recall from the computation with ranges at the beginning of the proof that $\Phi(H_{\bK,\fm}) \subseteq V_{\bL,\varphi(\fm)}$. Choose $\fn$ coprime to $\fm$, so also $\varphi(\fn)$ is coprime to $\varphi(\fm)$, so $y_{\fm}$ is zero on the support of $\varphi(\fn)$. Hence 
$$ \Phi(\gamma_{\fn}) = [(\vartheta_{\bL}(\varphi(\fn))^{-1} x_{\fm}, \varphi(\fn) y_{\fm})] = [(\vartheta_{\bL}(\varphi(\fn))^{-1} x_{\fm}, y_{\fm})] \in G_{\bL}^{\ab} \times \{ \Phi(\one_{\fm}) \}. $$ 
By density, we conclude that $$\Phi(H_{\bK,\fm}) = G_{\bL}^{\ab} \times_{\hat\cO_{\bL}^*} \{ \Phi(\one_{\fm}) \}. $$

By enlarging $\fm$, we find that $H_{\bK,\fm}\cong \mathring{G}_{\bK,\fm}^{\ab}$ is a system of exhausting quotient groups of $G_{\bK}^{\ab}$. 
Now observe that $\lim\limits_{N(\fm) \rightarrow +\infty} \one_{\fm} = 1$, so that the continuity of $\Phi$ implies that $\lim\limits_{N({\fm}) \rightarrow +\infty} \Phi(\one_{\fm}) = \Phi(1)$. 
We conclude that $\Phi$ induces a bijective map
$$ \Phi \colon G_{\bK}^{\ab} \times \{ 1 \} \rightarrow G_{\bL}^{\ab} \times \{ \Phi(1) \} $$ with the property that 
 $$ \Phi(1) \Phi(\gamma_{1} \gamma_{2})  = \Phi(\gamma_{1}) \Phi(\gamma_{2}). $$
If we set $\tilde{\Phi}(\gamma):=\Phi(\gamma) \cdot \Phi(1)^{-1},$ we find 
$$ \tilde{\Phi}(\gamma_1 \gamma_2) = \Phi(\gamma_1 \cdot \gamma_2) \Phi(1)^{-1} = \Phi(\gamma_1) \Phi(\gamma_2) \Phi(1)^{-2} = \tilde{\Phi}(\gamma_1) \cdot \tilde{\Phi}(\gamma_2), $$
so $\tilde{\Phi}$ is indeed a group isomorphism. 
\end{proof}

\begin{quote}
\textbf{Convention.} \emph{To simplify notations, we replace the original isomorphism of QSM-systems $\varphi$ (which is induced by the homeomorphism $\Phi^{-1}$ and the group isomorphisms $\varphi=\alpha_1$) by the QSM-isomorphism which is induced instead by the homeomorphism $\tilde{\Phi}^{-1}$ and the $\varphi=\alpha_1$, and from now on, we denote this new QSM-isomorphism by the same letter $\varphi$, so that for the associated $\tilde{\Phi}$, it holds that $\tilde{\Phi}=\Phi$.  } 
\end{quote}

\begin{cor} \label{imageonem}
For all $\fm \in J_{\bK}^+$, it holds true that $\Phi(\one_{\fm}) = \one_{\varphi(\fm)}$. 
\end{cor}

\begin{proof}
Set $\Phi(\one_{\fm}) = [(x_{\fm},y_{\fm})]$. Since $\Phi$ is a group isomorphism $H_{\bK} \rightarrow \Phi(H_{\bK})$, we find that $\Phi(\one_{\fm}^2)=\Phi(\one_{\fm})$; whence $$[(x^2_{\fm},y^2_{\fm})]=[(x_{\fm},y_{\fm} )], $$ i.e., there exists a unit $u \in \hat\cO_{\bL}^*$ with \begin{equation} \label{cancel} x_{\fm}^2 = \vartheta_{\bL}(u)^{-1} x_{\fm} \mbox{ and } y^2_{\fm} = u y_{\fm}.\end{equation}

Now $y_{\fm}$ is zero outside prime divisors of $\varphi(\fm)$. We claim that $y_{\fm}$ is a local unit at the primes dividing $\varphi(\fm)$.  If not, then 
$\Phi(\one_{\fm}) \in \mathrm{Range}(e_{\fr})$ for some prime ideal $\fr \in J_{\bL}^+$ which divides $\varphi(\fm)$. This is equivalent to the existence of $x\in X_{\bL}$ such that 
$\Phi(\one_{\fm})= \fr \ast x$.   This implies that $$\one_{\fm} = \Phi^{-1}(\fr \ast x) = \varphi^{-1}(\fr) \ast \Phi^{-1}(x).$$
 We conclude from this that $\one_{\fm} \in \mathrm{Range}(e_{\varphi^{-1}(\fr)})$. Now we observe that $\varphi^{-1}(\fr)$ is a prime ideal above a rational prime dividing $\fm$. In particular, it is not a unit at some prime divisor of $\fm$. But this contradicts the fact that all non-zero adelic components of $\one_{\fm}$ are such units. We conclude that $y_{\fm} \in \hat\cO_{\bL,\varphi(\fm)}^*$ is a unit. 

Hence in \eqref{cancel}, we can cancel $x_{\fm}$ (which lies in the group $G_{\bL}^{\ab}$) and $y_{\fm}$ locally at divisors of $\varphi(\fm)$, to find that 
$$ x_{\fm} = \vartheta_{\bL}(u)^{-1} \mbox{ and } y_{\fm} = u 1_{\varphi(\fm)}, $$
hence $$\Phi(\one_{\fm}) = [(x_{\fm},y_{\fm})] = [(\vartheta_{\bL}(u)^{-1}, u 1_{\varphi(\fm)})]=[(1,1_{\varphi(\fm)})] = \one_{\varphi(\fm)}. $$ 
\end{proof}

\section{Layers, ramification and $L$-series} \label{respect}

In this section, we conclude from the previous section that  $\varphi$ ``preserves ramification'', and we deduce from this that $\varphi$ induces an L-isomorphism (viz., an identification of abelian $L$-series). We will use the symbol $\Phi$ also for the group isomorphism that $\Phi \colon G_{\bK}^{\ab} \isomto G_{\bL}^{\ab}$ induces on quotient groups, i.e., if $N$ is a subgroup of $G_{\bK}^{\ab}$, then we let $\Phi$ also denote the isomorphism $$ G_{\bK}^{\ab} / N \isomto G_{\bL}^{\ab}/\Phi(N)$$ induced by $\Phi$. 

\begin{prop} \label{respectramification}
The group isomorphisms $ \Phi \, : \, G_{\bK}^{\ab} \isomto G_{\bL}^{\ab}$ and $\varphi \, : \, J_{\bK}^+ \isomto J_{\bL}^+$ respect ramification in the sense that if $\bK'=(\bK^{\ab})^N/\bK$ is a finite extension, and we set $\bL':=(\bL^{\ab})^{\Phi(N)}$ the corresponding extension of $\bL$, then 
$$ \p \mbox{ ramifies in } \bK'/\bK \iff \varphi(\p) \mbox{ ramifies in } \bL'/\bL $$
for every prime $\p \in J_{\bK}^+$. 
Hence
$$ \Phi(G_{\bK,\p}^{\ab}) = G_{\bL,\varphi(\p)}^{\ab} $$
for every prime $\p \in J_{\bK}^+$. 
\end{prop}

\begin{proof}
In the previous section, we saw that $\Phi$ induces an isomorphism
$$ \Phi \, : \, \mathring{G}_{\bK,\fn}^{\ab} \isomto \mathring{G}_{\bL,\varphi(\fn)}^{\ab},$$
of Galois groups of the maximal abelian extension $\bK_{\fn}$ that is unramified outside the prime divisors of $\fn$ and $\bL_{\varphi(\fn)}$ that is unramified outside $\varphi(\fn)$, respectively. 

Now let $\bK'=(\bK^{\ab})^N$ be a finite extension of $\bK$ ramified precisely above $$\p_1,\dots,\p_r \in J_{\bK}^+,$$ so $\bK' \subseteq \bK_{\p_1\cdots \p_r}$ and $$ \left\{ \begin{array}{ll} N \supseteq \mathrm{Gal}(\bK^{\ab}/\bK_{\p_1\dots \p_r}) & \\ N \not \supseteq \mathrm{Gal}(\bK^{\ab}/\bK_{\p_1\dots \widehat{\p_i} \dots \p_r}) & (i=1,\dots,r) \end{array} \right.$$ (where $\widehat\p$ means to leave out $\p$ from the product). 
Applying $\Phi$ and using the above result, we find that this is equivalent to 
$$ \left\{ \begin{array}{ll} \Phi(N) \supseteq \mathrm{Gal}(\bL^{\ab}/\bL_{\varphi(\p_1)\dots \varphi(\p_r)}) & \\ \Phi(N) \not \supseteq \mathrm{Gal}(\bL^{\ab}/\bL_{\varphi(\p_1)\dots \widehat{\varphi(\p_i)} \dots \varphi(\p_r)}) & (i=1,\dots,r) \end{array} \right.$$
 Thus, $\bL':=(\bL)^{\Phi(N)}$ is contained in $\bL_{\varphi(\p_1) \cdots \varphi(\p_r)}$ but not in any $\bL_{\varphi(\p_1) \cdots \widehat{\varphi(\p_i)}\cdots \varphi(\p_r)}$, and this means that $\bL'/\bL$ is ramified precisely above $\varphi(\p_1),\dots,\varphi(\p_r)$. 
\end{proof}

We now give a direct proof of the fact that (ii) implies (iii) in Theorem \ref{main2}. 

\begin{prop}\label{idLfunct}
An isomorphism $\varphi: (A_{\bK},\sigma_{\bK}) \to (A_{\bL}, \sigma_{\bL})$
induces an identification of $L$-series with characters, i.e., there is a group isomorphism of character groups $$\psi \, : \, \widehat{G}_{\bK}^{\ab} \isomto \widehat{G}_{\bL}^{\ab}$$ such that $$L_{\bK}(\chi,s)=L_{\bL}(\psi(\chi),s)$$ for all $\chi \in \widehat{G}_{\bK}^{\ab}$.
\end{prop}

\begin{proof}
By Proposition \ref{gab}, we have an isomorphism $\Phi \, : \, G_{\bK}^{\ab} \isomto G_{\bL}^{\ab}$, hence by Pontrjagin duality, an  identification of character groups $$ \psi \, : \, \widehat{G}_{\bK}^{\ab} \isomto \widehat{G}_{\bL}^{\ab}.$$ 

A character $\chi \in \widehat{G}_{\bK}^{\ab}$ extends to a function $f_\chi$ as in Section \ref{kmschar}. We claim that the function corresponding to $\psi(\chi)$ is  $\varphi(f_\chi)=f_{\psi(\chi)}$. To prove this, it suffices to check that divisors of the conductor $\f_{\psi(\chi)}$ of $\psi(\chi)$ are the same as divisors of $\varphi(\f_{\chi})$. But $\p$ is coprime to $\f_{\chi}$ precisely if $\chi$ factors over $G_{\bK,\p}^{\ab}$, and by the previous proposition, we find that this is equivalent to $\psi(\chi)=\Phi^*(\chi)$ factoring over $\Phi(G_{\bK,\p}^{\ab})=G_{\bL,\varphi(\p)}^{\ab}$, which in its turn means that $\varphi(\p)$ is coprime to the conductor $\f_{\psi(\chi)}$ of $\psi(\chi)$:
$$ (\p,\f_{\chi})=1 \iff (\varphi(\p),\f_{\psi(\chi)})=1. $$

The fact that $\varphi(f_\chi)=f_{\psi(\chi)}$ now implies that 
$$ \chi(\vartheta_{\bK}(\fn)) = \psi(\chi)(\vartheta_{\bL}(\varphi(\fn))) $$
for all $\chi \in \widehat{G}_{\bK}^{\ab}$ and $\fn \in J_{\bK}^+$ such that $\fn$ is coprime to the conductor of $\chi$. By the intertwining of time evolution, we also have compatibility with norms $$ N_{\bK}(\fn) = N_{\bL}(\varphi(\fn)) $$
for all $\fn \in J_{\bK}^+$. Hence we can compute
\begin{eqnarray*} L_{\bK}(\chi,s) &=& \sum_{\substack{\fn \in J_{\bK}^+\\(\fn,\f_{\chi})=1}} \frac{\chi(\vartheta_{\bK}(\fn))}{N_{\bK}(\fn)^s} = \sum_{\substack{\varphi(\fn) \in J_{\bL}^+\\(\fn,\f_{\chi})=1}} \frac{\psi(\chi)(\vartheta_{\bL}(\varphi(\fn)))}{N_{\bL}(\varphi(\fn))^s}\\ &=& \sum_{\substack{\fm \in J_{\bL}^+\\(\fm,\f_{\psi(\chi)})=1}} \frac{\psi(\chi)(\vartheta_{\bL}(\fm))}{N_{\bL}(\fm)^s} = L_{\bL}(\psi(\chi),s). \end{eqnarray*}
\end{proof}

\begin{remark}
The above result is a manifestation of the matching of $\KMS_\beta$ states. Namely, our isomorphism of QSM-systems gives $\zeta_{\bK}(\beta)=\zeta_{\bL}(\beta)$ (Proposition \ref{isotoareq}), and an isomorphism of character groups $\psi$ as in the previous proof. Lemma \ref{basic} implies that pullback is an isomorphism of $\KMS_\beta$-states. Now for $\beta>1$, such a state $\omega_{\gamma,\beta}^{\bL}$ on $A_{\bL}$ (corresponding to $\gamma \in G_{\bL}^{\ab}$) is pulled back to a similar state 
$$ \omega_{\gamma,\beta}^{\bL}(\varphi(f)) = \omega_{\tilde{\gamma},\beta}^{\bK}(f), $$
for some $\tilde{\gamma} \in G_{\bK}^{\ab}$ and every $f \in A_{\bK}$. We can choose in particular $f=f_\chi$ for a character $\chi \in \widehat G_{\bK}^{\ab}$, and then the above identity becomes
$$ \frac{1}{\zeta_{\bL}(\beta)\psi(\chi)(\gamma)}\ L_{\bL}(\psi(\chi),\beta) =  \frac{1}{\zeta_{\bK}(\beta)\chi(\tilde \gamma)}\ L_{\bK}(\chi,\beta). $$
If we now compare the constant coefficients and use arithmetic equivalence, we find  $ \psi(\chi)(\gamma)=\chi(\tilde \gamma)$, and so finally the identity of these particular $\KMS$-states indeed reads
$$ L_{\bL}(\psi(\chi),\beta) =  L_{\bK}(\chi,\beta). $$
\end{remark}

\section{From QSM-isomorphisms to isomorphism of unit ideles and ideles}

\begin{prop} \label{unitideles}
Let $\bK$ and $\bL$ denote two number fields admitting an isomorphism $\varphi$ of their QSM-systems $(A_{\bK},\sigma_{\bK})$ and $(A_{\bL},\sigma_{\bL})$. Let $\p \in J_{\bK}^+$ denote a prime ideal. Then $\varphi$ induces a group isomorphism of local units $$ \varphi \, : \, \hat\cO_{\bK,\p}^* \isomto \hat\cO_{\bL, \varphi(\p)}^* $$ and of unit ideles $$\varphi \, : \, \hat{\cO}^*_{\bK} \isomto \hat{\cO}^*_{\bL}.$$
\end{prop}

\begin{proof}
Consider the maximal abelian extension of $\bK$ in which $\p$ is unramified. It is the fixed field of the inertia group $I_{\bK,\p}^{\ab}$ of $\p$ in $\bK^{\ab}$. From the fact that $\Phi$ respects ramification, it follows that $$ \Phi(I_{\bK,\p}^{\ab}) = I_{\bL, \varphi(\p)}^{\ab}. $$
But now by local class field theory, we have a canonical isomorphism $$I^{\ab}_{\bK, \p} \isomto \hat\cO_{\bK, \p}^*.$$  Hence $\varphi$ induces isomorphisms 
$$ \varphi \, : \, \hat\cO_{\bK, \p}^* \isomto \hat\cO_{\bL, \varphi(\p)}^* $$
between the topological groups of local units  (compare with the discussion in Section 1.2 of \cite{Mo}). Since the integral invertible ideles are the direct product as topological groups of the local units, we get the claim. 
\end{proof}

\begin{prop} \label{finiteideles}
Let $\bK$ and $\bL$ denote two number fields admitting an isomorphism $\varphi$ of their QSM-systems $(A_{\bK},\sigma_{\bK})$ and $(A_{\bL},\sigma_{\bL})$. Then $\varphi$ induces a semigroup isomorphism:
$$ \varphi \, : \, (\A_{\bK,f}^* \cap \hat \cO_{\bK},\times) \isomto (\A_{\bL,f}^* \cap \hat \cO_{\bL},\times). $$
\end{prop}

\begin{proof}
We have an exact sequence \begin{equation} \label{sequence} 0 \rightarrow \hat{\cO}^*_{\bK} \rightarrow \A_{\bK,f}^* \cap \hat{\cO}_{\bK}  \rightarrow J^+_{\bK} \rightarrow 0,\end{equation} which is (non-canonically) split by choosing a uniformizer $\pi_{\p}$ at every place $\p$ of the field: 
$$ \A_{\bK,f}^* \cap \hat{\cO}_{\bK} \isomto J_{\bK}^+ \times \hat{\cO}^*_{\bK} \, : \, (x_{\p})_{\p} \mapsto \left( \prod \p^{\mathrm{ord}_{\p}(x_{\p})}, (x_{\p} \cdot \pi_{\p}^{-{\mathrm{ord}_{\p}(x_{\p})}})_{\p}    \right). $$ 
Hence as a semigroup, $\A_{\bK,f}^* \cap \hat \cO_{\bK}= J^+_{\bK}  \times \hat{\cO}^*_{\bK}. $ The result follows from Propositions \ref{gab} and \ref{unitideles}.
\end{proof}

\begin{remark}
Using fractional ideals, one may prove in a similar way that $\varphi$ induces a multiplicative group isomorphism of the finite ideles of $\bK$ and $\bL$. 
\end{remark}

\section{From QSM to field isomorphism: multiplicative structure}

In this section, we prove that QSM-isomorphism induces an isomorphism of multiplicative semigroups of rings of (totally positive) integers. The idea is to use certain symmetries of the system to encode this structure. 

We first establish some facts on the symmetries of QSM-systems of number fields. 
The statement is analogous to Proposition 2.14 of \cite{CMR} and Proposition 3.124 of \cite{CM}, where it was formulated for the case of imaginary quadratic fields, and to Theorem 2.14 of \cite{ConsM}, formulated in the function field case. 

\begin{prop} \label{bloeb}
Let $\bK$ denote any number field. An element $s$ of the semigroup $\hat{\cO}_{\bK} \cap \A^*_{\bK,f}$ induces an endomorphism $\varepsilon_{\bK,s}=\varepsilon_s$ of $(A_{\bK},\sigma_{\bK})$ given by $$ \varepsilon_s(f)(\gamma,\rho)= f(\gamma, s^{-1}\rho)e_{\fr} \quad \mbox{ and } \quad \varepsilon_s(\mu_{\fn})= e_{\fr} \mu_{\fn}, $$ where $e_{\fr}$ projects onto the space where 
$s^{-1}\rho \in \hat\cO_{\bK}$, for $\fr = s
\hat\cO_{\bK} \cap \bK.$ These endomorphisms preserve the dagger subalgebra $A_{\bK}^{\dagger}$ by construction. 

Furthermore, 
\begin{enumerate}

\item[\textup{(}i\textup{)}] The subgroup of invertible integral ideles $\hat{\cO}^*_{\bK}$ is exactly the one that acts by automorphisms of the system.
\item[\textup{(}ii\textup{)}] The closure of totally positive units $\bar{\cO_{\bK,+}^*}$ are precisely the elements that give rise to the trivial endomorphism.
\item[\textup{(}iii\textup{)}]  
The sub-semiring ${\cO}^\times_{\bK,+} = \cO_{\bK,+}-\{0\}$ of non-zero totally positive elements of the ring of integers is exactly the one that acts by dagger inner endomorphisms. 
\end{enumerate}
This is summarized by following commutative diagram:
$$ \xymatrix{  \mathrm{Inn}^\dagger(A_{\bK},\sigma_{\bK})  \ar@{^{(}->}[r]   & \mathrm{End}(A_{\bK},\sigma_{\bK}) & \mathrm{Aut}(A_{\bK},\sigma_{\bK}) \ar@{_{(}->}[l]   \\ \cO_{\bK,+}^\times \ar@{^{(}->}[r] \ar@{->}[u]  &  \hat{\cO}_{\bK} \cap \A^*_{\bK,f} \ar@{->}[u]_{\varepsilon_{\bK}} &   \hat{\cO}_{\bK}^*  \ar@{_{(}->}[l] \ar@{->}[u] \\ \cO_{\bK,+}^* \ar@{^{(}->}[r] \ar@{^{(}->}[u] &  \bar{\cO_{\bK,+}^*} \ar@{=}[r] \ar@{^{(}->}[u] &  \bar{\cO_{\bK,+}^*} \ar@{^{(}->}[u]  
}
$$
\end{prop}  

\begin{proof}
The maps $\varepsilon_s$ are indeed endomorphisms, since they are compatible by construction with
the time evolution, 
$$ \varepsilon_s \sigma_t = \sigma_t \varepsilon_s , \ \  \forall s\in \hat\cO \cap \A_{\bK,f}^*, \ \ 
\forall t\in \R. $$ We also see immediately that $\varepsilon_{\bK} \colon s \mapsto \varepsilon_s$ is a semigroup homomorphism. 

It is clear from the definition that exactly the elements of $\hat\cO_{\bK}^*$ act by automorphisms. 

An element $s$ acts trivially precisely when $(\gamma,\rho) \sim (\gamma,s^{-1}\rho)$ for all $\gamma,\rho$. This means that there exists an idelic unit $u \in \hat\cO_{\bK}^*$ such that $\vartheta_{\bK}(u)=1$ and $s=u$. Now class field theory says that $$\ker(\vartheta_{\bK}) \cap \hat\cO_{\bK}^*=\bar\cO_{\bK,+}^*,$$ the closure of the totally positive units of the ring of integers $\cO_{\bK}$ (compare Prop.\ 1.1 in \cite{LNT}). 

To finish the proof, we now study when $\varepsilon_s$ is an inner endomorphism that preserves the dagger subalgebra,
that is, an inner endomorphism implemented by an isometry $u \in A^{\dagger}_{\bK}$, which
is an eigenvector of the time evolution.
We claim the following:
\begin{quote}
\emph{If $\varepsilon_s(f)=u f u^*$ is a non-trivial dagger inner endomorphism for some eigenvector $u \in A^{\dagger}_{\bK}$ of the time evolution with $u^*u=1$, then, $u= a \mu_{\fr}$ for some phase factor $a \in C(X_{\bK})$ with $|a|^2=1$, and  
for some totally positive principal ideal $\fr \in J_{\bK}^+$. We then have 
$s \in \cO_{\bK,+}^\times$ with $\fr= s  \hat\cO_{\bK} \cap{\bK}$.}
\end{quote}

Indeed, suppose $u\in A^\dagger_{\bK}$
with $\sigma_t(u) =\lambda^{it} u$, for some 
$\lambda=n/m$ with $m,n$ coprime integers, and with $u^* u=1$.  
As an element in $A^{\dagger}_{\bK}$ the isometry $u$
can be written as  a sum of monomials
$$ u = \sum_{\fn} \mu_{\fn} f_{\fn} $$
with no $\mu_{\fn}^*$. Thus, it will have $m=1$ and $\lambda=n$ for
some $n$, so that
\begin{equation}\label{smuf1}
u=\sum_{N_{\bK}(\fn)=n} \mu_{\fn} f_{\fn}, 
\end{equation}
with $f_{\fn}\in C(X_{\bK})$. 

First observe the following: we can express all elements in the algebraic crossed product of $C(X_{\bK})$
by $J_{\bK}^+$ as sums of monomials of the form
$\mu_{\fn} f \mu^*_{\fm}$, with $\fn$ and $\fm$ in $J_{\bK}^+$ and $f\in C(X_{\bK})$.
For any pair of elements $\fn$ and $\fm$ in $J_{\bK}^+$ that have 
no factor in common in their decomposition into primes of ${\bK}$, 
let $V_{\fn,\fm}$ denote the linear span of the elements $\mu_{\fn} f \mu^*_{\fm}$
with $f\in C(X_{\bK})$.  Then $V_{\fn,\fm}\cap
V_{\fn',\fm'}=\{0\}$, whenever either $\fn \neq \fn'$ or $\fm\neq \fm'$. 

The condition $u^* u=1$ then gives
$$ \sum \bar f_{\fn} \mu_{\fn}^* \mu_{\fn'} f_{\fn'} = 1, $$
which we write equivalently as
\begin{equation}\label{twosums}
 \underbrace{\sum_{\fn} | f_{\fn}|^2}_{S_1} +\underbrace{\sum_{\substack{\fn \neq \fn'}}
\bar f_{\fn} \mu_{\fn}^* \mu_{\fn'} f_{\fn'}}_{S_2}  =1,
\end{equation}
where the first sum $S_1$ corresponds to the case where $\fn=\fn'$.

We now check that  the second sum $S_2$ vanishes. To see this, let $\fu$ be the greatest common factor of $\fn$ and $\fn'$ in their prime decompositions, so that $\fn=\fu\fa$ and $\fn'=\fu\fb$ for $\fa$ and $\fb$ coprime. 
Then we get
$$ \bar f_{\fn} \mu_{\fn}^* \mu_{\fn'} f_{\fn'}  =
\bar f_{\fn} \mu_{\fa}^* \mu_{\fb} f_{\fn'}, $$
since $\mu_{\fu}^* \mu_{\fu}=1$. Since $\fa$ and $\fb$ have no common factor,
$\mu_{\fa}^* \mu_{\fb}= \mu_{\fb}\mu_{\fa}^*$ and we have that the above expression further equals
$$ = \mu_{\fb} \sigma_{\fb}(\bar f_{\fn}) \sigma_{\fa}( f_{\fn'}) \mu_{\fa}^*. $$
Next notice that, since $\fa$ and $\fb$ have no common factor, this is an element of $V_{\fa,\fb}$.
Thus, in relation \eqref{twosums} the subsum $S_1$ and the constant $1$ on the right hand side are both in the subspace $V_{1,1}$, while all the terms in the second sum $S_2$ are in other
subspaces $V_{\fa,\fb}$ for $\fa \neq \fb$.  

We conclude that the second sum $S_2$ in \eqref{twosums} vanishes and thus, the condition that $u^*u=1$ is equivalent to the functions $f_{\fn}$ satisfying  
\begin{equation}\label{relm11}
 \sum_{\fn}  |f_{\fn}|^2 =1.
\end{equation}

Consider then the inner endomorphism $f\mapsto u f u^*$, with $u$ as above.
Substituting the above representation of $u$,  we find 
\begin{equation}\label{sumsufu}
 u f u^* = \sum_{\fn',\fn} \mu_{\fn} f_{\fn} f \bar f_{\fn'} \mu_{\fn'}^* =  \sum_{\fn', \fn} \rho_{\fn}( f_{\fn} f \bar f_{\fn'}) \mu_{\fn} \mu_{\fn'}^*
\end{equation}

As above, one verifies that the part of the above sum with $\fn' \neq \fn$ 
is in a space $V_{\fa,\fb}$ for $(\fa,\fb) \neq (1,1)$, while $u f u^* = \varepsilon_s(f)$ is in $V_{1,1}$, as is the part of the sum where $\fn'=\fn$, which equals 
$$ \sum_{\fn} \rho_{\fn}(|f_{\fn}|^2f)e_{\fn}. $$  

We conclude that 
\begin{equation}\label{esen}
  \varepsilon_s(f) =  f(\gamma,s^{-1}\rho)e_{\fr} =  \sum_{\fn}   a_{\fn} \rho_{\fn} ( f ) e_{\fn} . 
\end{equation}  
for $a_{\fn} = \rho_{\fn}( |f_{\fn}|^2)$ positive, supported in the range of $e_{\fn}$.  

Fix $\fn$ an ideal of norm $n$, different from $\fr$. We shall prove that $a_{\fn}=0$. For this, write $\fr=\fa \fb$ and $\fn = \fa \fc$ with $\fb$ and $\fc$ coprime. Assume that $a_{\fn}(x)\neq 0$. From the above, we can assume that $x$ belong to the range of $e_{\fn}=e_{\fa \fb^0 \fc}$. Assume by induction that $x$ belong to the range of $e_{\fa \fb^k \fc}$. We now show that $x$ also belongs to the range of $e_{\fa \fb^{k+1} \fc}$. For this, apply equation \eqref{esen} to the function $f=e_{\fb^k}$. We find 
$$ e_{\fr \fb^k}(x) = a_{\fn}(x) e_{\fn \fb^k}(x) + \mbox{ positive terms}. $$
We rewrite this as 
$$ e_{\fa \fb^{k+1}}(x) =  a_{\fn}(x) e_{\fa \fb^k \fc} (x) + \mbox{ positive terms}. $$
Since by assumption $a_{\fn}(x)>0$ and  $e_{\fa \fb^k \fc}(x)=1$, we find from this identity that $e_{\fa \fb^{k+1}}(x) \neq 0$. Hence $x$ belongs to the range of $\fa \fb^k \fc$ and $ \fa \fb^{k+1}$, hence of $\fa \fb^{k+1} \fc$ for all $k$, as claimed. If $\fb \neq 1$, then this never happens. We conclude that $\fb=1$, so $\fr = \fa \mid \fn$, and since $\fr$ and $\fn$ have the same norm, we find $a_{\fn}=0$ unless $\fn=\fr$, so that in the sum on the right hand side only one
term is non-zero, and relation \eqref{esen} becomes $$  \varepsilon_s(f) (\gamma,\rho) = \rho_{\fr}(|f_{\fr}|^2)\rho_{\fr}(f)(\gamma,\rho) e_{\fr}.$$  Working out both sides, we find 
\begin{equation}\label{epsilonsmun}
f(\gamma, \fr^{-1} \rho) e_{\fr} = \rho_{\fr}(|f_{\fr}|^2)
f(\theta_{\bK}(\fr)\gamma, \fr^{-1} \rho) e_{\fr}.
\end{equation}
First of all, setting $f=1$, we get that $\rho_{\fr}(|f_{\fr}|^2)=1$. If we apply the partial inverse $\sigma_{\fr}$ to this, we find $|f_{\fr}|^2 = \sigma_{\fr} \rho_{\fr} (|f_{\fr}|^2)=1$. We then conclude from \eqref{relm11} that all other $f_{\fn}=0$ ($\fn \neq \fr)$, so that we indeed get $$u= a \mu_{\fr}$$ for the phase factor $a=f_{\fr}$.  Now equality \eqref{epsilonsmun} implies that $\theta_{\bK}(\fr)$ equals $\theta_{\bK}(u)$
for some unit id\`ele $u \in \hat\cO_{\bK}^*$. This means precisely that $\fr$ is trivial in
$G_{\bK}^{\ab}/\vartheta_{\bK}(\hat\cO_{\bK}^*)$, which is the narrow ideal class group of ${\bK}$. Hence
$\fr$ is a totally positive principal ideal corresponding to a generator $s \in \cO^\times_{{\bK},+}$.
\end{proof}

\begin{remark}
As we have already observed, 
the group $G_{\bK}^{\ab}$ (which contains an image of $\hat\cO\cap
\A^*_{\bK,f}$) also acts on the QSM system by symmetries, cf. \cite{LLN},
Remark 2.2(i). 
This gives two slightly different actions of
$\hat\cO\cap \A^*_{\bK,f}$ on the QSM system, which induce the
same action on the low temperature KMS states. As was remarked to us by
Bora Yalkinoglu, when viewing the algebra $A_{\bK}$ as an endomotive
in the sense of \cite{CCM} and \cite{Mar-endo}, the two actions correspond,
respectively, to the one coming from the $\Lambda$-ring structure in
the sense of Borger \cite{Bor} and to the Galois action coming from
the endomotive construction as in \cite{CCM}.
\end{remark}

\begin{remark}[$\bK$-lattices] 
In terms of $\bK$-lattices $(\Lambda,\phi)$, the divisibility condition above corresponds to
the condition that the homomorphism $\phi$ factors through
$$\phi: \bK/\cO_{\bK} \to \bK\Lambda/\fn\Lambda \to \bK\Lambda/\Lambda.$$ The action of the endomorphisms
is then given by
$$ \varepsilon_s(f)((\Lambda,\phi),(\Lambda',\phi'))= f((\Lambda,s^{-1}\phi),(\Lambda',s^{-1}\phi')) $$
when both $(\Lambda,\phi)$ and $(\Lambda',\phi')$ are divisible by $s$ and zero otherwise.

When $s\in \cO_{\bK}^\times$, we can consider
the function 
$$ \mu_s ((\Lambda,\phi),(\Lambda',\phi'))= \left\{ \begin{array}{ll}
1 & \Lambda = s^{-1}\Lambda' \ \ \text{ and } \ \  \phi'=\phi; \\
0 & \text{otherwise.}
\end{array}\right. $$
These are eigenvectors of the time evolution, with 
$ \sigma_t(\mu_s) = N_{\bK} (\fn)^{it} \mu_s, $
and $ \varepsilon_s(f) = \mu_s \star f \star \mu_s^*, $
for the convolution product of the algebra $A_{\bK}$. For a discussion in this language of why, in the case of totally imaginary fields, only \emph{principal} (in this case, the same as totally positive principal) ideals give inner endomorphisms, see \cite{CM}, p.\ 562.
\end{remark}

\begin{prop} \label{integers}

Let $\bK$ and $\bL$ denote two number fields admitting a dagger isomorphism $\varphi$ of their QSM-systems $(A_{\bK},\sigma_{\bK})$ and $(A_{\bL},\sigma_{\bL})$. Then $\varphi$ induces a semigroup isomorphism between the multiplicative semigroups of totally positive non-zero elements of the rings of integers  of $\bK$ and $\bL$: 
$$ \varphi \, : \, (\cO^\times_{\bK,+},\times) \isomto (\cO^\times_{\bL,+},\times). $$
\end{prop}

\begin{proof} Proposition \ref{finiteideles} says that $\varphi$ induces an isomorphism
 $$ \varphi \, : \, \A_{\bK,f}^* \cap \hat{\cO}_{\bK} \isomto \A_{\bL,f}^* \cap \hat{\cO}_{\bL}. $$ 
From Proposition \ref{bloeb}, we have a map 
$$\varepsilon_{\bK} \, : \,  \A_{\bK,f}^* \cap \hat{\cO}_{\bK} \rightarrow \mathrm{End}(A_{\bK},\sigma_{\bK}) \colon s \mapsto \varepsilon_s $$ with kernel $\bar{\cO_{\bK,+}^*}$, and $\varphi$ induces a map 
$$ \mathrm{End}(A_{\bK},\sigma_{\bK}) \isomto \mathrm{End}(A_{\bL},\sigma_{\bL}). $$

Now $\varphi$, as an isomorphism of QSM-systems, also preserves the \emph{inner} endomorphisms: 
$$ \varphi \, : \, \mathrm{Inn}(A_{\bK},\sigma_{\bK}) \isomto \mathrm{Inn}(A_{\bL},\sigma_{\bL}). $$
Moreover, because the $C^*$-algebra isomorphism $\varphi$ also induces an isomorphism
of the dagger subalgebras $\varphi: A^{\dagger}_{\bK} \isomto A^{\dagger}_{\bL}$, it also preserves
the {\em dagger} inner endomorphisms,
$$ \varphi \, : \, \mathrm{Inn}^{\dagger}(A_{\bK},\sigma_{\bK}) \isomto \mathrm{Inn}^{\dagger}(A_{\bL},\sigma_{\bL}), $$
but we know that  $$\varepsilon_{\bK}^{-1}\left(\mathrm{Inn}^{\dagger}(A_{\bK},\sigma_{\bK})\right) = \cO_{\bK,+}^\times, $$ and similarly for $\bL$. Hence to prove that $\varphi$ gives an isomorphism 
\begin{equation} \label{PI} \varphi \, : \, \cO_{\bK,+}^\times \isomto \cO_{\bL,+}^\times, \end{equation} it suffices to prove that $\varphi$ maps $\varepsilon_{\bK}^{-1}\left(\mathrm{Inn}^\dagger(A_{\bK},\sigma_{\bK})\right)$ to $\varepsilon_{\bL}^{-1}\left(\mathrm{Inn}^\dagger(A_{\bL},\sigma_{\bL})\right).$ To prove this, we will verify that $\varphi \circ \varepsilon_{\bL} = \varepsilon_{\bK} \circ \varphi$, i.e., the commuting of the right square in the following diagram: 

$$ \xymatrix{ 
 & \mathrm{Inn}^\dagger(A_{\bK},\sigma_{\bK})  \ar@{^{(}->}[rr] \ar@{->}[ldd]^{\varphi}   & & \mathrm{End}(A_{\bK},\sigma_{\bK}) \ar@{->}[ldd]^{\varphi}    \\ 
 & \cO_{\bK,+}^\times \ar@{^{(}->}[rr] \ar@{->}[u]  \ar@{-->}[ldd]^{\varphi ?} &  & \hat{\cO}_{\bK} \cap \A^*_{\bK,f} \ar@{->}[u]_{\varepsilon_{\bK}} \ar@{->}[ldd]^{\varphi}   \\
 \mathrm{Inn}^\dagger(A_{\bL},\sigma_{\bL})  \ar@{^{(}->}[rr] & & \mathrm{End}(A_{\bL},\sigma_{\bL}) & \\ 
\cO_{\bL,+}^\times \ar@{^{(}->}[rr] \ar@{->}[u]  &  & \hat{\cO}_{\bL} \cap \A^*_{\bL,f} \ar@{->}[u]_{\varepsilon_{\bL}} & \\
 }
$$
This is equivalent to the following statement: 
\begin{lem} For every $s \in \A_{\bK,f}^* \cap \hat{\cO}_{\bK}$, we have that
$\varphi(\varepsilon_s) = \varepsilon_{\varphi(s)}.$
\end{lem} 

\begin{proof}
Since $\A_{\bK,f}^* \cap \hat{\cO}_{\bK}$ is isomorphic to the direct product of $\hat\cO_{\bK}^*$ and $J_{\bK}^+$, it suffices to prove this for these subgroups individually. Since the map $J_{\bK}^+ \rightarrow \mathrm{End}(A_{\bK},\sigma_{\bK})$ is injective, it is automatic that $\varepsilon_{\bK}$ and $\varphi$ intertwine with $\varepsilon_{\bL}$ on elements of this subgroup. Now suppose on the other hand that $s \in \hat\cO_{\bK}^*$. For a function $g \in C(X_{\bL})$, we have by definition 
\begin{eqnarray*} \varphi(\varepsilon_s)(g)(x)&=&(\varphi \circ \varepsilon_s \circ \varphi^{-1})g(x)\\ &=&g(\Phi((1,s^{-1})\cdot \Phi^{-1}(x))) \\ &=&  g(\Phi((\vartheta_{\bK}(s),1)\cdot y)),\end{eqnarray*}
where we have written $\Phi(y)=x$ for $y \in X_{\bK}$.

By the density statement in Corollary \ref{dude}, it suffices to compute this action on functions that are supported on $y=\fn \ast \gamma'$ for some $\gamma \in G_{\bK}^{\ab}$ and $\fn \in J_{\bK}^+$. But for such values, and any $\gamma \in G_{\bK}^{\ab}$, we have that 
$$\Phi(\gamma \cdot y) =  \Phi(\gamma \cdot (\fn \ast \gamma')) = \Phi(\fn \ast (\gamma \gamma')) = \varphi(\fn) \ast \Phi(\gamma\gamma'), $$
by Proposition \ref{gab}, and since $\Phi$ is multiplicative on elements in $G_{\bK}^{\ab}$ (Proposition \ref{gab}), we find that that this is further equal to  $$ \varphi(\fn) \ast \left(\Phi(\gamma) \Phi(\gamma')\right) =  \Phi(\gamma)\Phi(\fn \ast \gamma') = \Phi(\gamma)\Phi(y) . $$
We apply this with $\gamma=\vartheta_{\bK}(s)$ and $y=\Phi^{-1}(x)$, to find
\begin{eqnarray*} \varphi(\varepsilon_s)(g)(x) &=& g(\Phi((\vartheta_{\bK}(s),1)) \cdot x) \\ &=& g((1,\varphi(s)^{-1})\cdot x) \\ &=& \varepsilon_{\varphi(s)}(g)(x), 
\end{eqnarray*} which proves the statement. 
\end{proof}
With the proof of this lemma, we have reached the end of the proof of Proposition \ref{integers}. 
\end{proof}

\section{Recovering the additive structure}

\begin{se} In this section, we show that the map $\varphi$ is additive. For this, we prove it is additive (or, what is the same, the identity map)  modulo totally split primes. We do this by lifting elements of the residue field of a totally split prime to integers, which we show are fixed by the map $\varphi$. 

For the rest of this section, we assume that $\varphi \colon (A_{\bK},\sigma_{\bK}) \isomto (A_{\bL},\sigma_{\bL}) $ is a dagger isomorphism of QSM-systems of two number fields $\bK$ and $\bL$. Since $\bK$ and $\bL$ are arithmetically equivalent, they have the same discriminant, which we denote by $\Delta$. We choose a prime ideal $\p$ of $\bK$ of norm $p$, and let $\varphi(\p)$ denote the corresponding prime of $\bL$.
\end{se} 

\begin{notation} For an integer $N$, we let $\Z_{(N)}$ denote the set of integers coprime to $N$. We recall the following notations:  $1_{\p}=(0,\dots,0,1,0,\dots,0)$ denotes the adele with a $1$ at the $\p$-th place and $0$ everywhere else, and $\one_{\p}:=[(1,1_{\p})] \in X_{\bK}$. Finally, when $u \in \hat\cO_{\bK,\p}$, we let $u_{\p}$ denote the integral idele $u_{\p}:=(1,\dots,1,u,1,\dots,1)$, with $u$ in the $\p$-th place and $1$ everywhere else. 
\end{notation} 

\begin{se}
Recall that the map $\varphi \colon \hat\cO_{\bK,\p}^* \isomto \hat\cO_{\bL,\varphi(\p)}^*$ is constructed by canonically identifying both unit groups with the corresponding inertia groups in the maximal abelian extension, which are mapped to each other by the homomorphism $\Phi$. Said otherwise, for $u \in \hat\cO_{\bK,\p}^*$, the element $\varphi(u)$ is defined by $$ [(1,u_{\p})] = [(\vartheta_{\bK}(u_{\p})^{-1},1)] \mapsto \Phi([(\vartheta_{\bK}(u_{\p})^{-1}),1)]) = [(\Phi(\vartheta_{\bK}(u_{\p})^{-1}),1)] =:[(1,\varphi(u)_{\varphi(\p)})].$$
\end{se}

\begin{se}
We consider the composite map 
$$ \lambda_{\bK,\p} \colon \hat\cO_{\bK, \p}^* \rightarrow X_{\bK} \xrightarrow{[\cdot \one_{\p}]} X_{\bK} \colon u \mapsto [(1,u_{\p})] \mapsto [(1,u_{\p} \cdot 1_{\p})] = [(1,(0,\dots,0,u,0,\dots,0)]. $$
This is obviously a group isomorphism onto the image, which we denote by $Z_{\bK,\p}$. \end{se} 

\begin{lem} The following diagram commutes: 
$$ \xymatrix{   \hat\cO_{\bK,\p}^* \ar@{->>}[r]^{\lambda_{\bK,\p}} & Z_{\bK,\p} \\ 
\hat\cO_{\bL,\varphi(\p)}^* \ar@{->>}[r]_{\lambda_{\bL,\varphi(\p)}} \ar@{<-}[u]_{\varphi} & Z_{\bL,\varphi(\p)} \ar@{<-}[u]_{\Phi}
}$$
\end{lem}

\begin{proof} We need to verify that for any $u \in \hat\cO_{\bK,\p}^*$, it holds true that 
$$ \Phi([(1,u_{\p})] \cdot \one_{\p})=[(1,\varphi(u)_{\varphi(\p)})] \cdot \one_{\varphi(\p)}. $$ 
We compute that 
$$ [(1,u_{\p})] \cdot \one_{\p}=[(\vartheta_{\bK}(u_{\p})^{-1},1_{\p})], $$ which belongs to the group 
$$ H_{\bK} = G_{\bK}^{\ab} \times_{\hat \cO_{\bK}^*} \{ 1_{\p} \}, $$
which, as was shown in the proof of Proposition \ref{gab} and in Corollary \ref{imageonem}, is mapped by $\Phi$ to an element of the form 
$$ [(\Phi(\vartheta_{\bK}(u_{\p})^{-1}),1_{\varphi(\p)})] =  [1,\varphi(u)_{\varphi(\p)}] \cdot [(1,1_{\varphi(\p)})].$$
This proves the commutativity of the diagram. 
\end{proof}

\begin{lem}   Consider the map 
$$ \varpi_{\bK,\p} \colon \Z_{(p\Delta)} \hookrightarrow \hat\cO_{\bK,\p}^*  \rightarrow Z_{\bK,\p} \colon a \mapsto [(1,a\cdot 1_{\p})]$$
(and similarly for $\bL$). 
Then the map $\varpi_{\bK,\p}$ is injective, and the associated homeomorphism $\Phi$ is the identity map when restricted to the image of $\varpi$: we have a commutative diagram
$$ \xymatrix{ 
\Z_{(p\Delta)} \ar@{=}[d] \ar@/^12pt/[rr]^{\varpi_{\bK,\p}}  \ar@{->}[r] &\hat\cO^*_{\bK,\p} \ar@{->}[d]^{\varphi} \ar@{->>}[r]_{\lambda_{\bK,\p}} & Z_{\bK,\p} \ar@{->}[d]^{\Phi} \\
\Z_{(p\Delta)} \ar@/_12pt/[rr]_{\varpi_{\bL,\varphi(\p)}}  \ar@{->}[r] &\hat\cO^*_{\bL,\varphi(\p)} \ar@{->>}[r]^{\lambda_{\bL,\varphi(\p)}} & Z_{\bL,\varphi(\p)}  \\
}$$
where the curved arrows are injective. In particular, $\varphi \colon \hat\cO_{\bK,\p}^* \isomto \hat\cO_{\bL,\varphi(\p)}^*$ is constant on $\Z_{(p\Delta)}$.

\end{lem}

\begin{proof}

To prove the injectivity of $ \varpi_{\bK,\p}$, if $(1,a\cdot 1_{\p}) \sim (1,b\cdot 1_{\p})$ then there exists a unit $w \in \bar{\cO_{\bK,+}^*}$ with $a\cdot 1_{\p}=wb \cdot 1_{\p}$; hence $w \in \bQ \cap \bar{\cO_{\bK,+}^*} = \{1 \}$, so $a=b$. 

To prove the commutativity of the diagram, observe that for $a \in \Z_{(p\Delta)}$, we have 
$$ \varpi_{\bK,\p}(a) = (a) \ast [(\vartheta_{\bK}(a),1_{\p})] = (a) \ast [(1,1_{\p})] = (a)*\one_{\p}, $$
since $a \in \Z \subseteq \bK^*$ has trivial image under the reciprocity map. We compute the image by $\Phi$:
\begin{eqnarray*}
\Phi(\varpi_{\bK,\p} (a)) &=& \Phi((a) \ast \one_{\p})\\ 
&=& \varphi((a)) \ast \Phi(\one_{\p})\\
&=& (a) \ast \one_{\varphi(\p)}\\
&=& \varpi_{\bL,\varphi(\p)}(a)
\end{eqnarray*}
In this proof, we have used that $\varphi$ fixes the ideal $(a) \in J_{\bQ}^+$ for $a \in \Z_{(p\Delta)} $; by multiplicativity of $\varphi$, it suffices to prove this for $a$ a rational prime that is unramified in $\bK$ (viz., coprime to $\Delta$). Decompose such $(a)$ in $\bK$ as $(a)=\p_1\dots \p_r$ (with all $\p_i$ distinct, since $a$ is unramified). Since $\varphi=\alpha_1$ is a permutation of the distinct primes above the given rational prime $a$, we find that $\varphi((a))=\p_{\sigma(1)} \dots \p_{\sigma(r)}$ for some permutation $\sigma$ of the indices. Hence $\varphi((a))=(a)$, as desired. 

In the computation, we also used  that $\Phi(\one_{\p})=\one_{\varphi(\p)}$, which was shown in the previous lemma. 

Finally, the previous lemma (commutativity of the right square in the diagram) and the injectivity of the maps $\varpi$ on $\Z_{(p\Delta)}$ shows that the map $\varphi$ is the identity on $\Z_{(p\Delta)}$. 
\end{proof}

\begin{thm}
The map $\varphi \colon \cO^\times_{\bK,+} \isomto \cO^\times_{\bL,+}$, extended by $\varphi(0)=0$, is additive.
\end{thm}

\begin{proof}
Choose a rational prime $p$ that is totally split in $\bK$ (in particular, unramified). Then, since $\bK$ and $\bL$ are arithmetically equivalent, we have in particular that $p$ is also totally split in $\bL$. Choose a prime $\p \in J_{\bK}^+$ above $p$, so $f(\p|\bK)=1$; then $f(\varphi(\p)|\bL)=1$, too. 

From the map of localisations $ \varphi \colon \hat\cO^*_{\bK,\p} \isomto \hat\cO^*_{\bL,\varphi(\p)}$, we now construct a multiplicative map $\tilde{\varphi}$ of residue fields, using the Teichm\"uller lift $$\tau_{\bK,p} \colon \bar\bK^*_{p}\cong \F^*_p \hookrightarrow \hat\cO^*_{\bK,\p} \cong \Q^*_p$$ in the following diagram: 

$$ \xymatrix{ \hat\cO^*_{\bK,\p}\ar@{->}[r]^{\varphi} & \hat\cO^*_{\bL,\varphi(\p)}\ar@{->}[d]^{\mathrm{mod}\, \varphi(\p)}  \\ \bar\bK^*_{\p}\ar@{-->}[r]^{\tilde{\varphi}} \ar@{^{(}->}[u]^{\tau_{\bK,p}} & \bar\bL^*_{\varphi(\p)}}$$
The map $\tilde{\varphi}$ is multiplicative by construction. We will now prove that its extension by $\tilde{\varphi}(0)=0$ is additive (or, equivalently, $\tilde{\varphi} \colon \F^*_p \rightarrow \F^*_p$ is the identity map). 

We extend the Teichm\"uller character in the usual way to 
$$ \tau_{\bK,p} \colon \hat\cO^*_{\bK,\p} \rightarrow \hat\cO^*_{\bK,\p} \colon x \mapsto \lim_{n \rightarrow +\infty} x^{p^n}. $$

Now let $\tilde{a}$ denote any residue class in $\bar{\bK}_{\p}^* \cong \F_p$. Choose an integer $a$ that is congruent to $\tilde{a}$ mod $\p$ and coprime to the discriminant $\Delta$ (which is possible by the Chinese remainder theorem--- observe that $p$ and $\Delta$ are coprime). It holds true that $\tau_{\bK,p}(\tilde{a})=\tau_{\bK,p}(a)$ for the extended Teichm\"uller map. 
Since $\varphi$ is continuous in the $p$-adic topology and multiplicative, we find that 
\begin{eqnarray*} \varphi(\tau_{\bK,p}({a})) &=& \varphi \left( \lim_{n \rightarrow +\infty} a^{p^n} \right) \\ &=& \lim_{n \rightarrow +\infty} \varphi({a})^{p^n} \\ &=&  \tau_{\bL,p} ( \varphi({a} )) \\ &=& \tau_{\bL,p}({a}) \end{eqnarray*}
(the last equality follow from the lemma above), 
so that we find 
\begin{eqnarray*} \tilde{\varphi}(\tilde{a}) &=& \varphi(\tau_{\bK,p}(a)) \, \mathrm{mod}\, \varphi(\p)\\ &=& \tau_{\bL,p}(a)\, \mathrm{mod}\, \varphi(\p)\\ &=& \tilde{a} \, \mathrm{mod}\, \varphi(\p). \end{eqnarray*}

Hence $\varphi$ is the identity map modulo any totally split prime, so for any such  prime $\p \in J_{\bK}^+$ and any $x,y \in \cO_{\bK,+}$, we have 
$$ \varphi(x+y) = \varphi(x) + \varphi(y)\, \mathrm{mod}\, \varphi(\p).$$ Since there are totally split primes of arbitrary large norm (by Chebotarev), we find that $\varphi$ itself is additive. 
\end{proof}

\begin{thm}
Let $\bK$ and $\bL$ denote two number fields whose QSM-systems $(A_{\bK},\sigma_{\bK})$ and $(A_{\bL},\sigma_{\bL})$ are isomorphic. Then $\bK$ and $\bL$ are isomorphic as fields. 
\end{thm}

\begin{proof}
We have just seen that $\varphi$ induces an isomorphism of semigroups of totally positive integers (Proposition \ref{integers}). Now $\cO_{\bK}$ always has a free $\Z$-basis consisting of totally positive elements; indeed, if $y_1=1,y_2,\dots,y_n$ is any basis, replace it by $x_1=1,x_2=y_2+k_2,\dots,x_n=y_n+k_n$ where $k_i$ are integers with $k_i > - \sigma(y_i)$ for all real embeddings $\sigma$ of $\bK$. Then we can extend $\varphi \, : \cO_{\bK} \isomto \cO_{\bL}$ by 
$$ \varphi(\sum n_i x_i) \mapsto \sum_n n_i \varphi(x_i); $$
by the above this is well-defined, additive and multiplicative, and hence it extends further to an isomorphism of the quotient fields. 
\end{proof}

\part{$L$-SERIES AND QSM-ISOMORPHISM} \label{part2}
 
{Let $\chi$ denote a character in the Pontrjagin dual of $G_{\bK}^{\ab}$. We set 
$$ L_{\bK}(\chi,s) := \sum_{\fn \in J_{\bK}^+} \frac{\chi(\vartheta_{\bK}(\fn))}{N_{\bK}(\fn)^s}, $$
where it is always understood that we set $\chi(\vartheta_{\bK}(\fn))=0$ if $\fn$ is not coprime to the conductor $\f_{\chi}$ of $\chi$. 
This is also the Artin $L$-series for $\chi$ considered as a representation of the Galois group of the finite extension $\bK_\chi / \bK$ through which $\chi$ factors injectively (\cite{Neukirch}, VII.10.6).}

In the next few sections, we first show that (iii) $\Rightarrow$ (ii) in Theorem \ref{main2},
namely the identity of the $L$-functions implies the existence of a dagger isomorphism
of the quantum statistical mechanical systems, that is, a $C^*$-algebra isomorphism
$\varphi: A_{\bK} \isomto A_{\bL}$ intertwining the time evolutions, $\varphi\circ \sigma_{\bK}=
\sigma_{\bL}\circ \varphi$ and preserving the dagger subalgebras $\varphi: A^{\dagger}_{\bK}\isomto
A^{\dagger}_{\bL}$.

\section{QSM-isomorphism from matching $L$-series: compatible isomorphism of ideals}

\begin{prop}
Let $\bK$ and $\bL$ denote two number fields. Suppose $\psi$ is an isomorphism
$$ \psi \, : \, \widehat{G}^{\ab}_{\bK} \isomto \widehat{G}^{\ab}_{\bL}$$
that induces an identity of the respective $L$-functions
$$ L_{\bK}(\chi,s) = L_{\bL}(\psi(\chi),s). $$
Then there exists a norm preserving semigroup isomorphism 
$$\Psi \, : \, J_{\bK}^+ \rightarrow J_{\bL}^+, $$
which is compatible with the Artin reciprocity map under $\psi$ in the sense that 
\begin{equation} \label{NNN} \psi(\chi)(\vartheta_{\bL}(\Psi(\fn))) = \chi(\vartheta_{\bK}(\fn)) \end{equation}
for all characters $\chi$ and ideals $\fn$ such that the conductor of $\chi$ is coprime to $N_{\bK}(\fn)$ (which is also equivalent to (iv) in Theorem \ref{main3} for $\hat{\psi}:=(\psi^{-1})^*$). 
\end{prop}

\begin{proof} 
Since $\psi(1)=1$, the zeta functions ($L$-series for the trivial character) match on both sides: 
$$ \zeta_{\bK}(s)=\zeta_{\bL}(s).$$
This is arithmetic equivalence, and it shows in particular that there is a bijection between the sets of primes of $\bK$ and $\bL$ above a given rational prime $p$ and with a given inertia degree $f$. We need to match these primes in such a way that they are compatible with Artin reciprocity. We want to do this by mapping a prime $\p$ of $\bK$ to a prime $\q$ of $\bL$ above the same $p$, with the same inertia degree, and such that 
\begin{equation} \label{poeh} \psi(\chi)(\vartheta_{\bL}(\q)) = \chi(\vartheta_{\bK}(\p))  \end{equation} for all characters $\chi$ whose conductor is coprime to $p$. The main point is to show that it is always possible to find such $\q$, and to show that one may perform this in a bijective way between primes. We prove this by using a combination of $L$-series as counting function for the number of such ideals $\q$. 

The identification of $L$-series means that for any character $\chi$, we have 
\begin{equation} \label{tate} \sum_{\fn \in J_{\bK}^+} \frac{\chi(\vartheta_{\bK}(\fn))}{N_{\bK}(\fn)^s} =  \sum_{\fm \in J_{\bL}^+} \frac{\psi(\chi)(\vartheta_{\bL}(\fm))}{N_{\bL}(\fm)^s}. \end{equation}
We fix an integer $n$ and consider the norm-$n$ part of this identity: 
\begin{equation} \label{tate2} \sum_{{\fn \in J_{\bK}^+}\atop{N_{\bK}(\fn)=n}} {\chi(\vartheta_{\bK}(\fn))}=  \sum_{{\fm \in J_{\bL}^+}\atop{N_{\bL}(\fm)=n}} \psi(\chi)(\vartheta_{\bL}(\fm)). \end{equation}
In this notation, remember that we have set $\chi$ equal to zero on ideals not coprime to its conductor. 

Recall our notation $G_{\bK,\fn}^{\ab}$ for the Galois group of the maximal abelian extension of $\bK$ that is unramified above the prime divisors of an ideal $\fn$. We will take $\fn$ for the given integer $n$. 

We fix a finite quotient group $G$ of $$G_{\bK}^{\ab} \overset{\pi_G}{\twoheadrightarrow} G,$$ and consider only characters that factor over $G$, i.e., that are of the form $\chi \circ \pi_G$ for $\chi$ in the finite group $\widehat{G}$ (which we consider as a subgroups of $\widehat{G}_{\bK}^{\ab}$ by precomposing with $\pi_G$). We consider only $n$ that are coprime to the conductor of any character in $\widehat{G}$, so actually $\pi_G$ factors over $G_{\bK,n}^{\ab}$, and for such $n$, we sum the identity  (\ref{tate2}) over this group $\widehat{G}$, 
times the function $\chi(\pi_G(\gamma^{-1}))$ for a fixed element $\gamma \in G_{\bK}^{\ab}$ --- interchanging the order of summation, we find
 
\begin{equation} \label{integratedtate} \sum_{{\fn \in J_{\bK}^+}\atop{N_{\bK}(\fn)=n}}  \left(\sum_{\widehat{G}} \chi(\pi_G(\gamma)^{-1})\chi(\vartheta_{\bK}(\fn)) \right)   = \sum_{{\fm \in J_{\bL}^+}\atop{N_{\bL}(\fm)=n}}  \left(\sum_{\widehat{G}} \chi(\pi_G(\gamma)^{-1})\psi(\chi)(\vartheta_{\bL}(\fm)) \right).
 \end{equation}

Let us introduce the following set of ideals for $n \in \Z_{ \geq 1}$ and $\gamma \in G_{\bK}^{\ab}$:
$$ B_{G,n}(\gamma) = \{ \fn \in J_{\bK}^+ \, : \, N_{\bK}(\fn)=n \, \mbox{ and }  \, \pi_G(\vartheta_{\bK}(\fn))=\pi_G(\gamma) \}  $$
and denote the cardinality of this set by $$b_{G,n}(\gamma):=\#B_{G,n}(\gamma),$$ (or $b_{\bK,G,n}(\gamma)$ if we want to indicate the dependence on the ground field $\bK$). As is well-known, the value of the left hand side of Equation (\ref{integratedtate}) is 
$$\mbox{LHS}(\ref{integratedtate}) =  {|G|} \cdot b_{\bK,G,n}(\gamma). $$

We now perform a base change in the bracketed sum on the right hand side of (\ref{integratedtate}), using the homomorphism $$\psi \, : \, \widehat{G}_{\bK}^{\ab} \rightarrow 
\widehat{G}_{\bL}^{\ab},$$ which we can do since $\psi$ preserves the subgroups indexed by $n$:
$$ \psi(\widehat G_{\bK,n}^{\ab}) = \widehat G_{\bL,n}^{\ab}. $$
Indeed, if $\f_\chi$ is not coprime to $n$, $L_{\bK}(\chi,s)$ has a missing Euler factor at a prime number $p$ dividing $n$. Hence, by the equality of $L$-series, also $L_{\bL}(\psi(\chi),s)$ has such a missing Euler factor, so $\f_{\psi(\chi)}$ is not coprime to $p$ (hence $n$).  

To ease notation we write $(\psi^{-1})^*(G)=G'$. We also write $\eta=\psi(\chi)$. Then the bracketed expression on the right hand side of  (\ref{integratedtate}) becomes
\begin{equation} \label{RHS2} \sum_{\widehat{G'}} \psi^{-1}(\eta)(\pi_G(\gamma)^{-1})\eta(\pi_{G'}(\vartheta_{\bL}(\fm))) 
 \end{equation}
Observe that for fixed $\fm$ coprime to $\f_\eta$,  $$\Xi_{\fm} \, : \, \eta \mapsto \psi^{-1}(\eta)(\pi_G(\gamma)^{-1})\eta(\pi_{G'}(\vartheta_{\bL}(\fm)))$$ is a character on $\widehat{G'}$. Thus, 
$$ \sum_{\widehat{G'}} \psi^{-1}(\eta)(\pi_G(\gamma)^{-1})\eta(\pi_{G'}(\vartheta_{\bL}(\fm))) = \left\{ \begin{array}{ll} {|G'|} & \mbox{ if }\ \Xi_{\fm} \equiv 1; \\ 0 & \mbox{ otherwise.}\end{array} \right.$$
Now $\Xi_{\fm}\equiv 1$ means that 
$$ \eta(\pi_{G'}(\vartheta_{\bL}(\fm))) = \psi^{-1}(\eta)(\pi_{G}(\gamma)) \mbox{ for all } \eta \in G'. $$
Since the right expression is equal to $\eta(\pi_{G'}((\psi^{-1})^* \gamma))$, we find that $\Xi_{\fm} \equiv 1$ means that 
$$\pi_{G'}(\vartheta_{\bL}(\fm)) = \pi_{G'}((\psi^{-1})^*(\gamma)).$$
Plugging everything back in, we find that the right hand side of Equation (\ref{integratedtate}) becomes
\begin{eqnarray*}  \mbox{RHS}(\ref{integratedtate}) &=&  {|G'|} \cdot \# \{\fm \in J_{\bL}^+ \mbox{ with } N_{\bL}(\fm)=n \mbox{ and }\pi_{G'}(\vartheta_{\bL}(\fm))=\pi_{G'}((\psi^{-1})^*(\gamma)) \}  \\ &=&  {|G'|} \cdot  b_{\bL,G',n}((\psi^{-1})^*(\gamma)). \end{eqnarray*}

Since $\psi$ is a group isomorphism of finite abelian groups, $|G'|=|\widehat{G'}|=|\psi(\widehat{G})|=|G|$, so we conclude that for all finite quotient groups $G$ of $G_{\bK,n}^{\ab}$ 
\begin{equation} \label{count}  b_{\bK,G,n}(\gamma) = b_{\bL,(\psi^{-1})^*G,n}((\psi^{-1})^*(\gamma)). \end{equation}
Now as a profinite group, $G_{\bK,n}^{\ab}$ can be written as the inverse limit over all its finite quotients, and since all constructions are compatible with these limits, we conclude that the sets
\begin{equation} S_1(n,\gamma):=\{ \fn \in J_{\bK}^+ \, : \, N_{\bK}(\fn)=n \, \mbox{ and }  \, \pi_{G_{\bK,n}^{\ab}}(\vartheta_{\bK}(\fn)))=\pi_{G_{\bK,n}^{\ab}}(\gamma) \} 
\end{equation}
and
\begin{equation} S_2(n,\gamma):=\{\fm \in J_{\bL}^+ \mbox{ with } N_{\bL}(\fm)=n \mbox{ and }\pi_{G_{\bL,n}^{\ab}}(\vartheta_{\bL}(\fm))=\pi_{G_{\bL,n}^{\ab}}((\psi^{-1})^*(\gamma)) \} 
\end{equation}
have the same number of elements. We now set $\gamma=\vartheta(\tilde\fn)$ for a given ideal $\tilde\fn \in J_{\bK}^+$ of norm $n$. Since the Artin map $\vartheta_{\bK} \, : \, J_{\bK}^+ \rightarrow G_{\bK,n}^{\ab}$ is injective on ideals that divide $n$, we find that the set $S_1(N_{\bK}(\tilde\fn),\vartheta_{\bK}(\tilde\fn))$ has a unique element. Hence there is also a unique ideal $\fm \in J_{\bL}^+$ with 
$$ N_{\bL}(\fm) = N_{\bK}(\tilde\fn) $$
and
\begin{equation} \label{MMM} \pi_{G_{\bL,n}^{\ab}}(\vartheta_{\bL}(\fm)) = \pi_{G_{\bL,n}^{\ab}}((\psi^{-1})^* ( \vartheta_{\bK}(\tilde\fn))). \end{equation} After applying Pontrjagin duality, this becomes exactly statement (\ref{NNN}). 
We set $\Psi(\tilde\fn):=\fm$, and this is our desired map. It is multiplicative, since $(\psi^{-1})^*$ and the Artin maps are so. 

Finally, \eqref{NNN} is equivalent to (iv) in Theorem \ref{main3} (``Reciprocity isomorphism'') for $\hat \psi = (\psi^{-1})^*$, since the latter statement is clearly equivalent to \eqref{MMM}. \end{proof}

\section{QSM-isomorphism from matching $L$-series: homeomorphism on $X_{\bK}$}

 We now proceed to show that $\psi$ also induced a natural map on the whole abelian part $C(X_{\bK}) \rightarrow C(X_{\bL})$, not just on the part $ \psi \, : \, C(G_{\bK}^{\ab}) \isomto C(G_{\bL}^{\ab})$ where it is automatically defined (by continuity of $\psi$). We check this on ``finite'' parts of these algebras that exhaust the whole algebra, as in Section \ref{respect}.  
 \begin{lem} The map $\psi$ extends to an algebra isomorphism
 $$ \psi \, : \, C(G_{\bK}^{ab}\times_{\hat\cO_{\bK}^*} \hat \cO_{\bK}) \rightarrow C(G_{\bL}^{ab}\times_{\hat\cO_{\bL}^*} \hat \cO_{\bL}) .$$
 \end{lem}
 
 \begin{proof} Recall that the map $\psi \, : \, \widehat G_{\bK}^{\ab} \isomto \widehat G_{\bL}^{\ab}$ induces by duality a group isomorphism 
 $$ (\psi^{-1})^* \, : \, G_{\bK}^{\ab} \isomto G_{\bL}^{\ab}, $$ and let $\Psi \, : \, J_{\bK}^+ \isomto J_{\bL}^+$ denote the compatible isomorphism of semigroups of ideals introduced in the previous section.  
 
Recall from Section \ref{respect} how we have decomposed the algebra $C(X_{\bK})$ into pieces $C(X_{\bK,n})$, were we now assume $n$ is an integer. We can then define a map 
$$ \psi_n \, : \, C(X_{\bK,n}) \rightarrow C(X_{\bL,n}) $$
as the closure of the map given by $$f_{\chi,\fm}\mapsto f_{\psi(\chi),\Psi(\fm)},$$ where $f_{\chi,\fm}$ are the generators of the algebra $C(X_{\bK,n})$ given in Lemma \ref{gengen}.  Recall from the previous section that if $\chi \in \widehat G_{\bK}^{\ab}$ has conductor coprime to $n$, so has $\psi(\chi)$, so the map $\psi_n$ is well-defined.  The map is a vector space isomorphism by construction, since both $\psi$ and $\Psi$ are bijective. 

By taking direct limits (the maps of algebras are compatible with the divisibility relation), we arrive at a topological vector space isomorphism
$$ \psi = \lim_{{\longrightarrow}\atop{n}} \psi_n \, : \, C(X_{\bK}) \isomto C(X_{\bL}). $$

To see that the map $\psi$ is an algebra homomorphism, we need to check it is compatible with multiplication: this will follow from the compatibility of $\Psi$ with the Artin map, which implies that the function $\psi(f_{\chi,\fm})$ is given by a pullback. Indeed, for $x = \Psi(\fm') \ast [(\gamma',\rho')] \in X_{\bL,n}^1$ with $\gamma' \in G_{\bL,n}^{\ab}, \rho' \in \hat\cO_{\bL,n}^*$ and $\fm' \in J_{\bK,n}^+$, we find  that 
\begin{eqnarray*} \psi(f_{\chi,\fm})(x) &=& f_{\psi(\chi),\Psi(\fm)}(\Psi(\fm') \ast [(\gamma',\rho')]) \\ &=& \delta_{\Psi(\fm),\Psi(\fm')}  \psi(\chi)(\vartheta_{\bL}(\Psi(\fm)^{-1}) \psi(\chi)(\gamma'), \end{eqnarray*}
which, by the compatibility of $\Psi$ with the reciprocity map (Equation \eqref{NNN}), is 
$$ = \delta_{\fm,\fm'} \chi(\vartheta_{\bK}(\fm)^{-1}) \chi(\psi^*(\gamma')) = (\psi^{-1})^* f_{\chi,\fm}(x). $$
Hence if $\chi$ and $\chi'$ are two characters in $\widehat G_{\bK}^{\ab}$, and $\fm, \fm'$ are two ideals in $J_{\bK,n}^+$ for $n $ sufficiently large, we find
$$ \psi(f_{\chi,\fm} \cdot f_{\chi',\fm'}) = (\psi^{-1})^* \left( f_{\chi,\fm} \cdot f_{\chi',\fm'} \right)= (\psi^{-1})^* \left(f_{\chi,\fm}\right) \cdot (\psi^{-1})^*  \left( f_{\chi',\fm'} \right) = \psi(f_{\chi,\fm}) \cdot \psi(f_{\chi',\fm'}), $$
which implies that $\psi$ is multiplicative. 

\end{proof}

\section{QSM-isomorphism from matching $L$-series: end of proof}

\begin{thm}\label{matchLiso}
Let $\bK$ and $\bL$ denote two number fields. Suppose $\psi$ is a group isomorphism
$$ \psi \, : \, \widehat{G}^{\ab}_{\bK} \isomto \widehat{G}^{\ab}_{\bL}$$
that induces an identity of the respective $L$-functions
$$ L_{\bK}(\chi,s) = L_{\bL}(\psi(\chi),s). $$
Then there is a dagger isomorphism of QSM-systems
$\varphi: (A_{\bK},\sigma_{\bK}) \to (A_{\bL}, \sigma_{\bL})$.
\end{thm}

\begin{proof} 
The maps $\psi \colon X_{\bK} \rightarrow X_{\bL}$ and $\mu_{\fn} \mapsto  \mu_{\Psi(\fn)}$ induce an isomorphism 
$$ \varphi: A^{\dagger}_{\bK} \rightarrow A^{\dagger}_{\bL}, $$
which extends to a $C^*$-algebra isomorphism between $A_{\bK}$ and $A_{\bL}$. 

It remains to verify that this map is indeed a QSM-isomorphism, i.e., that it commutes with time evolution. On the abelian part, there is nothing to verify, since it is stable by time evolution. On the semigroup part, it is a simple consequence of the fact that $\Psi$ preserves norms: 
$$ N_{\bL}(\Psi(\fn)) = N_{\bK}(\fn), $$
so that, on the one hand 
$$\sigma_{\bL,t}(\varphi(\mu_{\fn})) = N_{\bL}(\Psi(\fn))^{it} \mu_{\Psi(\fn)}, $$ and on the other hand, 
$$\varphi(\sigma_{\bK,t}(\mu_{\fn})) = \varphi( N_{\bK}(\fn)^{it} \mu_n) = N_{\bK}(\fn)^{it} \mu_{\Psi(\fn)}.$$
This finishes the proof that $$\sigma_{\bL,t} \circ \varphi = \varphi \circ \sigma_{{\bK},t}.$$ 
\end{proof}

\begin{remark}
As quoted in the introduction, in \cite{Riem}, it was shown that an equality of infinitely many Dirichlet series associated to a map between closed Riemannian manifolds is equivalent to this map being an isometry. In the same reference, it is then shown how to use this theorem to define a distance between closed Riemannian manifolds, as infimum over a usual distance between complex functions. With number fields, we are now in a very analogous situation, in that we characterize number fields by an equality of Dirichlet series. One might use this to define a distance on the set of all number fields up to isomorphism.  It then remains to investigate whether this (forcedly discrete) distance on a countable set has an interesting completion (much like passing from $\Q$ to $\R$): are there interesting `limits' of number fields? Also, metrizing abstract number fields might be useful to our understanding of ``arithmetic statistics''--- the distribution of invariants over number fields (compare \cite{Venk}). \end{remark}

\section{Proof of Theorem \ref{main3}}

We now show that  reciprocity isomorphism (iv) implies L-isomorphism (iii). Since obviously, field isomorphism (i) implies reciprocity isomorphism (iv), this will finish the proof of all main theorems from the introduction.   

The condition of compatibility with Artin maps at finite level can be rephrased as follows: for any $\fn$ dividing an integer $n$, we have that $$\pi_{G_{\bK,n}^{\ab}} (\hat \psi(\vartheta_{\bK}(\fn))) = \pi_{G_{\bL,n}^{\ab}}(\vartheta_{\bL}(\Psi(\fn))), $$ to which we can apply Pontrjagin duality to find that 
$$ \chi(\vartheta_{\bK}(\fn)) = \psi(\chi)(\vartheta_{\bL}(\Psi(\fn)), $$
for all characters $\chi$ whose conductor is coprime to $n$. Here, we define the map $\psi$ by 
$$ \psi=(\hat\psi^{-1})^* \colon \widehat G_{\bK}^{\ab} \isomto \widehat G_{\bL}^{\ab}. 
$$
Let $\chi \in \widehat G_{\bK}^{\ab}$. We prove the theorem by performing a change in summation $\fm = \Psi(\fn)$ in the $L$-series as follows (using that norms are preserved, and Artin maps intertwined): 

$$ L_{\bK}(\chi,s) = \sum_{\fn \in J_{\bK}^+} \frac{\chi(\vartheta_{\bK}(\fn))}{N_{\bK}(\fn)^s} = \sum_{\fm \in J_{\bL}^+} \frac{\psi(\chi)(\vartheta_{\bL}(\fm))}{N_{\bL}(\fm)^s} = L_{\bL} (\psi(\chi),s). $$
(Recall that in this definition, we have set $\chi(\vartheta_{\bK}(\fn))=0$ as soon as $\fn$ is not coprime to the conductor of $\chi$.)

\begin{remark}
In Uchida's proof of the function field case of the Neukirch-Uchida theorem (\cite{U}), the construction of a multiplicative map of global function fields $(\bK^*,\times) \isomto (\bL^*,\times)$ is based on the existence of topological group isomorphisms of the ideles $ \Psi \, : \, \A_{\bK}^* \isomto \A_{\bL}^*$ and of the abelianized Galois groups $ \hat\psi \, : \, G_{\bK}^{\ab} \isomto G_{\bL}^{\ab}$ which are compatible with the Artin maps, using that in a function field $\bK$, the group $\bK^*$ is the kernel of the Artin map (which is not surjective in this case). The conditions that go into this proof are a bit similar to the ones in Theorem \ref{main3}. Our theorem shows that similar conditions imply the same result for number fields as for function fields, albeit with a rather different proof.
\end{remark} 

\begin{remark} Around 1992, Dinakar Ramakrishnan asked whether isomorphism between two number fields $\bK$ and $\bL$ is equivalent to the existence of an isomorphism $\alpha \colon \A_{\bK} \isomto \A_{\bL}$ of their respective adele rings and an isomorphism $\omega \colon W_{\bK}^{\ab} \isomto W_{\bL}^{\ab}$ of the abelianizations of their Weil groups. If these two isomorphisms are compatible with reciprocity in the sense that the following diagram commutes
$$  \xymatrix{ \A_{\bK}^* \ar@{->}[r]^{\alpha} & \A_{\bL}^* \\  W_{\bK}^{\ab} \ar@{->}[r]^{\omega} \ar@{->}[u] & W_{\bL}^{\ab}  \ar@{->}[u] },$$
then their kernels are isomorphic, so $\alpha$ restricts to an isomorphism $\bK^* \isomto \bL^*$, which, extended by $0 \mapsto 0$, gives a field isomorphism of $\bK$ and $\bL$ (the additivity is automatic from the embedding into the adele rings). The question remains whether the same holds without assuming compatibility of the maps via reciprocity. 
\end{remark}

\section{Relaxing the conditions on $L$-series}

\begin{se}
One may now wonder whether condition (iii) (L-isomorphism) of Theorem \ref{main2}, can be weakened. For example, is it possible to restrict to characters of fixed type? At least for rational characters of order two (i.e., arising from quadratic extensions by the square root of a rational number), this is not the case, as the following proposition shows. 
\end{se} 

\begin{prop}
Suppose $\bK$ and $\bL$ are number fields with the same Dedekind zeta function. Then for any quadratic character $\chi$ whose conductor is a rational non-square in $\bK$ nor $\bL$, we have an equality of $L$-series $L_{\bK}(\chi,s)=L_{\bL}(\chi,s)$. \end{prop}

\begin{proof}
We have \begin{equation} \label{1} \zeta_{\bK}(s) = \zeta_{\bL}(s) \end{equation} 
This says that $\bK$ and $\bL$ are arithmetically equivalent, which we can express in group theoretical terms by Ga{\ss}mann's criterion (\cite{Perlis1}) as follows: let $\bN$ be Galois over $\Q$ containing $\bK$ and $\bL$; then $\Gal(\bN/\bK)$ and $\Gal(\bN/\bL)$ intersect all conjugacy classes in $\Gal(\bN/\Q)$ in the same number of elements. 

Let $\bM=\Q(\sqrt{d})$ for a rational non-square $d$. It is easy to see from Ga{\ss}mann's criterion for arithmetic equivalence that then, the composita $\bK \bM$ and $\bL \bM$ are also arithmetically equivalent (cf.\ e.g.\ Uchida \cite{U2}, Lemma 1): choose $\bN$ so it also contains $\bM$, and verify that  $\Gal(\bN/\bK\bM)$ and $\Gal(\bN/\bL\bM)$ intersect all conjugacy classes in $\Gal(\bN/\Q)$ in the same number of elements. 
We conclude that the zeta functions of $\bK \bM = \bK(\sqrt{d})$ and $\bL=\bL(\sqrt{d})$ are equal:
 \begin{equation} \label{2} \zeta_{\bK\bM}(s) = \zeta_{\bL\bM}(s) \end{equation} 
Let $\chi$ be the quadratic character that belongs to $d$. By Artin factorization, we can write \begin{equation} \label{3} \zeta_{\bK \bM}(s) = \zeta_{\bK}(s) \cdot L_{\bK}(\chi,s) \mbox{ and } \zeta_{\bL \bM}(s) = \zeta_{\bL}(s) \cdot L_{\bL}(\chi,s). \end{equation}

We find the conclusion by combining (\ref{1}), (\ref{2}) and (\ref{3}).
\end{proof}

\begin{remark}
We do not know a direct ``analytic'' proof that equality of zeta functions implies equality of all rational quadratic twist $L$-series. As a matter of fact, looked at in a purely analytic way, the result does not appear to be so obvious at all. 
\end{remark}

\begin{remark}
Bart de Smit \cite{BDS} has proven that for $\bK$ and $\bL$ to be isomorphic, it suffices to have an equality between the sets of all zeta functions of abelian extensions of $\bK$ and $\bL$, or between the sets of all $L$-series for characters of order $\leq 2$. His method is constructive in the sense that, for given arithmetically equivalent $\bK$ and $\bL$, one may construct a finite set of quadratic characters whose $L$-series have to match for $\bK$ and $\bL$ to be isomorphic. 
\end{remark}

\begin{remark}
One may wonder how much information an equality of \emph{sets} of $L$-series with characters encodes about the characters themselves (so not assuming the identification of $L$-series to arise from an isomorphism of abelianized Galois groups).  Bart de Smit has constructed an example of two number fields and two characters of \emph{different} order whose $L$-series coincide. Multiplicative relations (more general than equality) between $L$-series on the same number field are discussed in \cite{Artin} (around \emph{Satz} 5). 
 \end{remark}

\bibliographystyle{amsplain}

\end{document}